\documentclass[11pt, reqno]{amsart}
\usepackage[hmargin=3cm,vmargin=3cm]{geometry}                
\usepackage{graphicx} 
\usepackage{amssymb} 
\usepackage{epstopdf}
\usepackage[all]{xy}
\usepackage{hyperref}
\usepackage{mathtools}
\usepackage{xifthen}

\numberwithin{equation}{section}

\newcommand{\mc}[1]{\mathcal{#1}}
\newcommand{\id}{{\rm id}}
\newcommand{\<}{\langle}
\renewcommand{\>}{\rangle}
\newcommand{\Aut}{{\rm Aut}}
\newcommand{\spn}{{\rm span}}
\newcommand{\wot}{{\rm wot}}

\newcommand{\ten}[1][]{\ifthenelse{\equal{#1}{}}{\,\ol{\otimes}\,}{\,\ol{\otimes}_{#1}\,}}
\newcommand{\sten}[1][]{\ifthenelse{\equal{#1}{}}{\otimes}{\otimes_{#1}}}
\newcommand{\tpi}{\tilde\pi}
\renewcommand{\and}{\quad \text{and} \quad}
\newcommand{\ol}[1]{\overline{#1}}
\newcommand{\on}[1][]{\stackrel{#1}{\curvearrowright}}
\newcommand{\vN}[1]{\{#1\}''}
\let\oldcite\cite
\renewcommand{\cite}[2][]{\ifthenelse{\equal{#1}{}}{\hyperref[#2]{\oldcite{#2}}}{\hyperref[#2]{\oldcite[#1]{#2}}}}
\let\oldref\ref
\renewcommand{\ref}[1]{\hyperref[#1]{\oldref{#1}}}
\let\oldeqref\eqref
\renewcommand{\eqref}[1]{\hyperref[#1]{\oldeqref{#1}}}
\newcommand{\Tr}{{\rm Tr}}

\newcommand{\B}{\mc{B}}
\newcommand{\CC}{\mathbb{C}}
\newcommand{\E}{\text{E}}
\renewcommand{\H}{\mc{H}}
\newcommand{\K}{\mc{K}}

\newcommand{\MHM}{{_M\H_M}}
\newcommand{\MHN}{{_M\H_N}}
\newcommand{\MKN}{{_M\K_N}}

\newcommand{\NKP}{{_N\K_P}}
\newcommand{\N}{\mc{N}}
\newcommand{\NN}{\mathbb{N}}

\newcommand{\RR}{\mathbb{R}}
\newcommand{\R}{\mc{R}}

\newcommand{\U}{\mc{U}}

\newcommand{\Z}{\mathcal{Z}}
\newcommand{\ZZ}{\mathbb{Z}}

\makeatletter
\newtheorem*{rep@theorem}{\rep@title}
\newcommand{\newreptheorem}[2]{%
\newenvironment{rep#1}[1]{%
 \def\rep@title{#2 \ref{##1}}%
 \begin{rep@theorem}}%
 {\end{rep@theorem}}}
\makeatother

\theoremstyle{definition}

\newreptheorem{mytheorem}{Theorem}
\newreptheorem{myprop}{Proposition}
\newreptheorem{mycor}{Corollary}
\theoremstyle{plain}
\newtheorem{mytheorem}{Theorem}

\newtheorem{myprop}[mytheorem]{Proposition}
\newtheorem{mycor}[mytheorem]{Corollary}
\newtheorem{theorem}{Theorem}[section]

\newtheorem{prop}[theorem]{Proposition}
\newtheorem{lemma}[theorem]{Lemma}

\title[Equivalence Relations with Nontrivial One-Cohomology]{Von Neumann Algebras of Equivalence Relations with Nontrivial One-Cohomology}
\author{\vspace{-.2cm}Daniel J. Hoff}
\address{Mathematics Department; University of California, San Diego, CA 90095-1555 (United States).}
\email{d1hoff@ucsd.edu}
\thanks{This material is based upon work supported by the National Science Foundation Graduate Research Fellowship Program under Grant No. DGE-1144086.}

\begin{document} 

\maketitle
\vspace{-1cm}
\begin{abstract}
Using Popa's deformation/rigidity theory, we investigate prime decompositions of von Neumann algebras of the form $L(\mathcal{R})$ for countable probability measure preserving equivalence relations $\mathcal{R}$.
We show that $L(\mathcal{R})$ is prime whenever $\mathcal{R}$ is nonamenable, ergodic, and admits an unbounded 1-cocycle into a mixing orthogonal representation weakly contained in the regular representation. This is accomplished by constructing 
the \emph{Gaussian extension $\tilde{\mathcal{R}}$ of $\mathcal{R}$} and subsequently an s-malleable deformation of the inclusion $L(\mathcal{R}) \subset L(\tilde{\mathcal{R}})$. 
We go on to note a general obstruction to unique prime factorization, and avoiding it, we prove a unique prime factorization result for products of the form $L(\mathcal{R}_1) \otimes L(\mathcal{R}_2) \otimes \cdots \otimes L(\mathcal{R}_k)$. As a corollary, we get a unique factorization result in the equivalence relation setting for products of the form $\mathcal{R}_1 \times \mathcal{R}_2 \times \cdots \times \mathcal{R}_k$.
We finish with an application to the measure equivalence of groups. 
\end{abstract}
\vspace{.7cm}

\section{Introduction}

\subsection{Background and statement of results.}
A natural question in the classification of von Neumann algebras asks how a tracial von Neumann algebra can be written as the tensor product of subalgebras.
A tracial von Neumann algebra $M$ is called \emph{prime} if whenever $M = N \ten Q$ for subalgebras $N, Q \subset M$, either $N$ or $Q$ is of type ${\rm I}$. For ${\rm II}_1$ factors $M$, this amounts to forcing either $N$ or $Q$ to be finite dimensional. A ${\rm II}_1$ factor is called \emph{solid} if the relative commutant of any diffuse subalgebra is amenable. All non-amenable subfactors of a solid ${\rm II}_1$ factor are prime.  

In \cite{Po83}, Popa proved primeness for certain ${\rm II}_1$ factors with non-separable preduals, including the group von Neumann algebra of the free group on uncountably many generators. 
Then in \cite{Ge96}, using free probability theory, Ge showed that the free group factors $L(\mathbb{F}_n)$ are prime as well. 
In \cite{Oz03}, Ozawa used $C^*$-algebraic methods to prove that $L(\Gamma)$ is in fact solid for all icc hyperbolic groups $\Gamma$, recovering the primeness of $L(\mathbb{F}_n)$ as a special case. 
By developing a new technique of closable derivations, Peterson showed in \cite{Pe06} that $L(\Gamma)$ is prime for nonamenable icc groups which admit an unbounded 1-cocycle into a multiple of the left regular representation. 
Popa then used his powerful deformation/rigidity theory to give a new proof of solidity for $L(\mathbb{F}_n)$, \cite{Po06c}. 
Using Sinclair's malleable deformation of $L(\Gamma)$ arising from an unbounded 1-cocycle \cite{Si10}, Vaes showed in \cite{Va10b} that deformation/rigidity theory could also be used to recover Peterson's result. In this paper, we construct an analogous deformation of $L(\R)$ and use Popa's theory to prove the following analogue of Peterson's primeness result in the setting of countable pmp equivalence relations:

\begin{mytheorem}\label{T: main}
Let $\R$ be a countable pmp equivalence relation with no amenable direct summand which admits an unbounded 1-cocycle into a mixing orthogonal representation weakly contained in the regular representation. Then $L(\R) \ncong N \ten Q$ for any type ${\rm II}$ von Neumann algebras $N$ and $Q$ and
hence $\R \ncong \R_1 \times \R_2$ for any pmp $\R_i$ which have a.e. equivalence class infinite.
In particular, if $\R$ is ergodic then $L(\R)$ is prime. 
\end{mytheorem}

For additional primeness results, we refer the reader to \cite{Oz04, Po06a, CI08, CH08, Bou12, DI12}. For a ${\rm II}_1$ factor which is not prime, it is natural to ask if it can be written uniquely as the tensor product of prime subfactors. Of course, if $M = P_1 \ten P_2$ for prime ${\rm II}_1$ factors $P_1$ and $P_2$, then any $u \in \U(P_1 \ten P_2)$ gives $M = uP_1u^* \ten uP_2u^*$ as a prime factorization of $M$. Moreover, for any ${\rm II}_1$ factors $N, Q$ and $t > 0$, there is a natural identification $N \ten Q \cong N^t \ten Q^{1/t}$, where $N^t$ denotes the amplification of $N$ by $t$ (see Section \ref{SS: amplifications}). Hence prime factorization results are considered up to such amplification as well as up to unitary conjugacy. 

In fact, as first proved by Ozawa and Popa in \cite{OP03} and subsequently in \cite{Pe06, CS11, SW11, Is14, CKP14, HI15}, the techniques used to prove primeness can often be used to prove unique prime factorization results. However, we find that in the setting of $L(\R)$, the presence of the Cartan subalgebra $L^\infty(X) \subset L(\R)$ can present additional obstacles to passing from a primeness result to a unique prime factorization result. These obstacles do not appear to have been encountered before; to best of our knowledge this paper gives the first unique prime factorization result for factors of the form $L(\R)$ (or $L^\infty(X) \rtimes \Lambda$) that do not arise also as $L(\Gamma)$ for some countable group $\Gamma$. 

The root of the difficulties in the setting of $L(\R)$ lies in the fact that our $s$-malleable deformation of $L(\R)$ does not deform the Cartan subalgebra $L^\infty(X)$.
As an example, take any free ergodic action of a nonabelian free group $\mathbb{F}_n$ on a standard probability space $(X, \mu)$. Then the orbit equivalence relation $\R = \R(\mathbb{F}_n \on X)$ will satisfy the assumptions of Theorem \ref{T: main}, so that $P = L(\R)$ is prime. But if we now assume that the action of $\mathbb{F}_n$ is not strongly ergodic, then $P$ will have property Gamma\footnote{A ${\rm II}_1$ factor $M$ has \emph{property Gamma} if there exists a sequence of unitaries $\{u_n\} \subset M$ with $\tau(u_n) = 0$ for all $n$ and $\|u_nx - xu_n\|_2 \to 0$ for each $x \in M$.} and the following theorem shows that $P \ten P$ admits two prime factorizations which are distinct up to unitary conjugacy and amplification:


\begin{mytheorem}\label{T: necfactor}
Let $M_1$ and $M_2$ be $\|\cdot\|_2$-separable $\rm{II}_1$ factors with property Gamma and set $M = M_1 \ten M_2$. Then there is an approximately inner automorphism $\phi \in \ol{\text{Inn}(M)}$ such that $\phi(M_i) \nprec M_j$ for any $i, j \in \{1, 2\}$. 
\end{mytheorem}

In particular, this implies that there is no $t > 0$, $i, j \in \{1, 2\}$ such that $\phi(M_i)$ is unitarily conjugate to $M_j^t$ in $M$. To avoid this obstruction, when considering unique factorization we will restrict to the case of strongly ergodic $\R$ and use Popa's deformation rigidity theory to prove the following:

\begin{myprop}\label{P: nonGamma}
Let $\R$ be a strongly ergodic countable pmp equivalence relation which is nonamenable and admits an unbounded 1-cocycle into a mixing orthogonal representation weakly contained in the regular representation. Then $L(\R)$ is prime and does not have property Gamma. 
\end{myprop}

Still, the presence of $L^\infty(X) \subset L(\R)$ presents additional difficulty in applying the techniques developed in \cite{OP03}. Nevertheless, we are able to prove the following:

\begin{mytheorem}\label{T: factorR}
For $i \in \{1, 2, \dots, k\}$, let $\R_i$ be a nonamenable strongly ergodic countable pmp equivalence relation which admits an unbounded 1-cocycle into a mixing orthogonal representation weakly contained in the regular representation. 
Then for each $i$, $L(\R_i)$ is prime and does not have property Gamma, and  

(1). If $M = L(\R_1) \ten L(\R_2) \ten \dots \ten L(\R_k) = N \ten Q$ for tracial factors $N, Q$, there must be a partition $I_N \cup I_Q = \{1, \dots, k\}$ and $t > 0$ such that  
$N^t = \bigotimes_{i \in I_N} L(\R_i)$ and $Q^{1/t} = \bigotimes_{i \in I_Q} L(\R_i)$ modulo unitary conjugacy in $M$.

(2). If $M = L(\R_1) \ten L(\R_2) \ten \dots \ten L(\R_k) = P_1 \ten P_2 \ten \cdots \ten P_m$ for $\rm{II}_1$ factors $P_1, \dots, P_m$ and $m \ge k$, then $m = k$, each $P_i$ is prime, and there are $t_1, \dots, t_k > 0$ with $t_1t_2\cdots t_k = 1$ such that after reordering indices and conjugating by a unitary in $M$ we have $L(\R_i) = P_i^{t_i}$ for all $i$. 

(3). In (2), the assumption $m \ge k$ can be omitted if each $P_i$ is assumed to be prime.
\end{mytheorem}

As an application, we prove the following corollary:

\begin{mycor}\label{C: factorRR}
Let $\R_1, \R_2, \dots, \R_k$ be as in Theorem \ref{T: factorR}. 

(1). If $\R_1 \times \R_2 \times \dots \times \R_k \cong S_1 \times S_2$ for infinite pmp equivalence relations $S_1$ and $S_2$, then there is $t > 0$ and an integer $1 \le m < k$ such that after reordering the indices we have
$S_1^t \cong \R_1 \times \R_2 \times \dots \times \R_m$ and $S_2^{1/t} \cong \R_{m+1} \times \R_{m+2} \times \dots \times \R_k$.

(2). If $\R_1 \times \R_2 \times \dots \times \R_k \cong S_1 \times S_2 \times \cdots \times S_m$ for infinite pmp equivalence relations $S_1, S_2, \dots, S_m$ and $m \ge k$, then $m = k$ and there are $t_1, \dots, t_k > 0$ with $t_1t_2\cdots t_k = 1$ such that after reordering indices we have $\R_i \cong S_i^{t_i}$ for all $i$. 
\end{mycor}

Note that in Theorem \ref{T: factorR} we assume that each $\R_i$ is strongly ergodic,
but that the obstruction in Theorem \ref{T: necfactor} only applies when multiple factors have property Gamma. We leave open the case of exactly one factor with Gamma.

We conclude with an application to the measure equivalence of countable groups. In \cite{Ga02}, Gaboriau showed that measure equivalent groups have proportional $\ell^2$ Betti numbers. It follows that a countable group with positive first $\ell^2$ Betti number cannot be measure equivalent to a product of infinite groups. The following theorem augments this conclusion:  

\begin{mytheorem}\label{T: ME}
Let $\Gamma$ be a countable nonamenable group which admits an unbounded 1-cocycle into a mixing orthogonal representation weakly contained in the left regular representation. Then $\Gamma \stackrel{\text{ME}}{\nsim} \Gamma_1 \times \Gamma_2$ for any infinite groups $\Gamma_1, \Gamma_2$. 
\end{mytheorem}


\subsection{Organization and strategy} Following the introduction, we establish the necessary preliminaries in Section \ref{S: prelim}. In Section \ref{S: deduce} we review how an s-malleable deformation can be used to prove primeness, condensing this strategy as Theorem \ref{T: deduce}. Section \ref{S: deform} then constructs such a deformation of $L(\R)$ by considering the Gaussian extension $\tilde \R$ of $\R$, and Section \ref{S: main} combines this construction with Theorem \ref{T: deduce} to prove primeness for $L(\R)$, Theorem \ref{T: main}. 

In Section \ref{S: factor}, we go on to apply this strategy in the more general context of prime factorization. We first prove the obstruction in Theorem \ref{T: necfactor}, then condense the general strategy as Theorem \ref{T: factor}. Proving Proposition \ref{P: nonGamma} allows us to apply this strategy to prove Theorem \ref{T: factorR} and subsequently Corollary \ref{C: factorRR}. The paper concludes in Section \ref{S: ME} with the application to the measure equivalence of groups, Theorem \ref{T: ME}. 
\\

{\it Acknowledgements.} I would like to extend my warm thanks to Adrian Ioana for proposing the topics of study in this paper, and for the many invaluable discussions without which it would not exist. I am very grateful to him for all of his shared wisdom and unfailing encouragement in every aspect of the process. I would also like to thank Stefaan Vaes and Remi Boutonnet for their very helpful remarks. Finally, I would like to express my gratitude for the detailed and insightful suggestions of the referee and to Alessandro Carderi for an illuminating discussion of them.  

\vspace{.6cm}
\section{Preliminaries} \label{S: prelim}
Throughout, $M, N, P$ and $Q$ will denote tracial von Neumann algebras, which we will always take to be $\|\cdot\|_2$-separable. We will let $\tau$ denote the trace on each where there is no danger of confusion, and the unit ball of (say) $M$ will be written as $(M)_1$. The group of unitary operators in $M$ will be denoted $\U(M)$, and if $N \subset M$, then $\N_M(N) = \{u \in U(M) : uNu^* = N\}$ will denote the normalizer of $N$ in $M$. We will write $e_N \in \B(L^2(M))$ for the orthogonal projection onto $L^2(N)$, and $E_N: M \to N$ will denote the resulting faithful normal conditional expectation onto $N$. 

\subsection{Measured Equivalence Relations} 

We review here the foundations of the study of measured equivalence relations as established by Feldman and Moore in \cite{FM75a}. Throughout, let $(X, \mu)$ denote a standard probability space. A \emph{measured equivalence relation} on $(X, \mu)$ is an equivalence relation $\R$ on $X$ such that $\R \subset X \times X$ is measurable in the product space. For $x \in X$, let $[x]_\R$ denote the $\R$-equivalence class of $x$. 
$\R$ is called \emph{countable} if $[x]_\R$ is countable (or finite) for a.e. $x \in X$.

We denote by $[\R]$ the \emph{full group} of $\R$, that is, $[\R] = \{\phi \in \Aut(X) : \text{graph}( \phi ) \subset \R\}$ where we write $\Aut(X)$ for the group of bimeasurable bijections on $X$. 
$\R$ is \emph{probability measure preserving (pmp)} if $\mu \circ \phi = \mu$ for all $\phi \in [\R]$.
A pmp $\R$ is \emph{ergodic} if $\mu(E) \in \{0, 1\}$ for any measurable $E \subset X$ satisfying $\mu(E\setminus\phi(E)) = 0$ for all $\phi \in [\R]$, and \emph{strongly ergodic} if $\mu(E_n)(1-\mu(E_n)) \to 0$ for any sequence of measurable subsets $E_n \subset X$ satisfying $\mu(E_n \setminus \phi(E_n)) \to 0$ for each $\phi \in [\R]$.  

Given a positive measure subset $E \subset X$, we denote by $\R|_{E}$ the measured equivalence relation on the probability space $(E, \mu/\mu(E))$ given by $\R|_E = \R \cap (E \times E)$. 
Measured equivalence relations $\R_1$ on $(X_1, \mu_1)$ and $\R_2$ on $(X_2, \mu_2)$ are \emph{isomorphic}, written $\R_1 \cong \R_2$, if there are full measure subsets $E_1 \subset X_1$, $E_2 \subset X_2$ which admit a measure space isomorphism $\phi: (E_1, \mu_1|_{E_1}) \to (E_2, \mu_2|_{E_2})$ such that 
\[(x, y) \in \R_1|_{E_1} \iff (\phi(x), \phi(y)) \in \R_2|_{E_2}.\]

Henceforth, $\R$ will always denote a countable pmp equivalence relation on a standard probability space $(X, \mu)$. We endow $\R$ with a measure $m$ given by 
\[
m(E) = \int_X |\{y \in [x]_\R: (x, y) \in E\}| d\mu(x) \quad \text{for all measurable} \quad E \subset \R
\]

\subsection{Equivalence Relation von Neumann Algebras}
To such an equivalence relation $\R$, we associate a von Neumann algebra $L(\R)$, first constructed and studied by Feldman and Moore in \cite{FM75b}. 
Each $g \in [\R]$ gives rise to a unitary $u_g \in \U(L^2(\R, m))$ defined by $[u_g f](x, y) = f(g^{-1}x, y)$. Similarly, each $a \in A = L^\infty(X)$ is identified with an operator in $\B(L^2(\R, m))$ by $[af](x, y) = a(x)f(x, y)$. The von Neumann algebra $L(\R)$ of the equivalence relation $\R$ is defined to be
\[
L(\R) = \vN{L^\infty(X), \{u_g : g \in [\R]\}} \subset \B(L^2(\R, m))
\]

$L(\R)$ has a faithful normal trace given by $\tau(x) = \<x 1_D, 1_D\>$, where $1_D \in L^2(\R, m)$ is the characteristic function of the diagonal $D = \{(x, x) : x \in X\}$. We note that $L^2(L(\R), \tau) \cong L^2(\R, m)$ as $L(\R)$ modules and we will identify these Hilbert spaces henceforth. 

If $\R$ is ergodic then $L(\R)$ is a factor, and if $\R$ is strongly ergodic then any sequence $\{a_n\} \subset (A)_1$ with $\|a_nu_g - u_ga_n\|_2 \to 0$ for each $g \in [\R]$ must have $\|a_n - \tau(a_n)\|_2 \to 0$. 

Note that $L^\infty(\R, m)$ acts on $L^2(\R, m)$ by pointwise multiplication and that $L^\infty(\R, m)$ is normalized by each unitary $u_g$ with $g \in [\R]$. 
Recall that $\R$ is called \emph{amenable} if there is a state $\Phi$ on $L^\infty(\R, m)$ with $\Phi(u_g f u_g^*)= \Phi(f)$ for all $f \in L^\infty(\R)$, $g \in [\R]$ and such that $\Phi|_{L^\infty(X)} = \tau$. One can show that $L(\R)$ is an amenable von Neumann algebra if and only if $\R$ is amenable. We say $\R$ has an amenable direct summand if there is a measurable subset $Y \subset X$ such that $\mu(Y) > 0$, $\R|_Y$ is amenable, and $\R = \R|_Y \cup \R|_{Y^c}$. In this case, $L(\R) = L(\R|_Y) \oplus L(\R|_{Y^c})$ has an amenable direct summand as well.

Let $Z^1(\R, S^1)$ denote the group of $S^1$-valued \emph{multiplicative 1-cocycles} on $\R$, that is, the group of measurable maps $c: \R \to S^1$ such that for $\mu$-a.e. $x \in X$, 
\begin{equation} \label{E: multcocycle}
c(x, y)c(y, z) = c(x, z) \quad \text{for all $(x, y), (y, z) \in \R$,}
\end{equation}
identifying cocycles that agree $m$-a.e.  
Given $c \in Z^1(\R, S^1)$ and $g \in [\R]$, let $f_{c, g} \in \U(L^\infty(X))$ be given by $f_{c, g}(x) = c(x, g^{-1}x)$. Then using \eqref{E: multcocycle}, one can check that the formula 
\begin{equation} \label{E: auto}
\psi_c(au_{g}) = f_{c, g}au_{g} \quad \text{for $a \in L^\infty(X), g \in [\R]$}
\end{equation}
gives rise to a well defined $*$-isomorphism $\psi_c \in \Aut(L(\R))$. Note that $\psi_{c_1}\circ\psi_{c_2} = \psi_{c_1c_2}$, so $c \mapsto \psi_c$ defines an action $\psi: Z^1(\R, S^1) \to \Aut(L(\R))$.

\subsection{Representations of Equivalence Relations}

In analogy to group representations on Hilbert spaces, pmp equivalence relations on $X$ are represented on measurable Hilbert bundles with base $X$. For an excellent detailed account of measurable Hilbert bundles, we refer the reader to \cite{Di69}. We recall here a few of the necessary facts. 

Given a collection of Hilbert spaces $\{\H_x\}_{x \in X}$, we form the \emph{Hilbert bundle} $X \ast \mc{H}$ as the set of pairs $X \ast \mc{H} = \{(x, \xi_x) : x \in X, \xi_x \in \H_x\}$. 
A section $\xi$ of $X \ast \H$ is a map $x \mapsto \xi(x) \in \mc{H}_x$. 

A \emph{measurable Hilbert bundle} is a Hilbert bundle $X \ast \H$ endowed with a $\sigma$-algebra generated by the maps $\{(x, \xi_x) \mapsto \<\xi_x, \xi_n(x)\>\}_{n = 1}^\infty$ for a \emph{fundamental sequence of sections} $\{\xi_n\}_{n = 1}^\infty$ satisfying 

$(i)$ $\H_x = \ol{\spn \{\xi_n(x)\}_{n = 1}^\infty}$ for each $x \in X$, and 

$(ii)$ the maps $\{x \mapsto \|\xi_n(x)\|\}_{n = 1}^\infty$ are measurable. 

It is a useful fact that the $\sigma$-algebra of any measurable Hilbert bundle can be generated by an \emph{orthonormal fundamental sequence of sections}, i.e. sections which moreover satisfy 

$(iii)$ $\{\xi_n(x)\}_{n = 1}^{\infty}$ is an orthonormal basis of $\H_x$ for $x \in X$ with $\dim \H_x = \infty$, and if $\dim \H_x < \infty$, the sequence $\{\xi_n(x)\}_{n = 1}^{\dim \H_x}$ is an orthonormal basis and $\xi_n(x) = 0$ for $n > \dim \H_x$. 
\\

A \emph{measurable section} of $X \ast \H$ is a section $\xi$ such that $x \mapsto (x, \xi(x)) \in X \ast \H$ is a measurable map, or equivalently, such that the maps $\{x \mapsto \<\xi(x), \xi_n(x)\>\}_{n = 1}^\infty$ are measurable for the fundamental sequence of sections $\{\xi_n\}_{n = 1}^{\infty}$. We let $S(X \ast \H)$ denote the vector space of measurable sections, identifying $\mu$-a.e. equal sections. We then consider the \emph{direct integral} 
\begin{align*}
\int_X^{\oplus} \H_x d\mu(x) = \{\xi \in S(X \ast \H) : \int_X \|\xi(x)\|^2 d\mu(x) < \infty\}
\end{align*}
which is a Hilbert space with inner product $\<\xi, \eta\> = \int_X \<\xi(x), \eta(x)\> d\mu(x)$. If $a \in A = L^\infty(X)$ and $\xi \in \int_X^{\oplus} \H_x d\mu(x)$ we denote by $a\xi$ or $\xi a$ the element of $\int_X^{\oplus} \H_x d\mu(x)$ given by $[a\xi](x) = [\xi a](x) = \xi(x)a(x)$. If $\{\xi_n\}_{n = 1}^{\infty}$ is an orthonormal fundamental sequence of sections, any $\xi \in \int_X^{\oplus} \H_x d\mu(x)$ has an expansion $\xi = \sum_{n = 1}^\infty a_n\xi_n$ where $a_n = \<\xi(\cdot), \xi_n(\cdot)\> \in A$. 

A \emph{unitary (resp. orthogonal) representation} of $\mc{R}$ on a complex (real) measurable Hilbert bundle $X \ast \H$ is a map $(x, y) \mapsto \pi(x, y) \in \U(\H_y, \H_x)$\footnote{For complex (resp. real) Hilbert spaces $\H, \K$, we write $\U(\H, \K)$ for the set of unitary (orthogonal) maps from $\H$ onto $\K$} on $\R$ such that for $\mu$-a.e. $x \in X$, we have 
\[\pi(x, y)\pi(y, z) = \pi(x, z) \quad \text{for all} \quad (x, y), (y, z) \in \R,\]
and such that $(x, y) \mapsto \<\pi(x, y)\xi(y), \eta(x)\>$ is a measurable map on $\R$ for all $\xi, \eta \in S(X \ast \H)$. 

Given a measurable Hilbert bundle $X \ast \H$ with an orthonormal fundamental sequence of sections $\mc{S} = \{\xi_n\}_{n = 1}^\infty$, we can always form the \emph{identity representation} $\id_S$ of $\R$ on $X \ast \H$, where $\id_{\mc{S}}(x, y)$ is determined by the formula $\id_{\mc{S}}(x, y)\xi_n(y) = \xi_n(x)$ for each $(x, y) \in \R$, $\xi_n \in \mc{S}$. 

To define the regular representation of $\R$, take $\H_x = \ell^2([x]_\R)$ for each $x \in X$, and form the measurable Hilbert bundle $X \ast \H$ with fundamental sequence of sections $\{\xi_g\}_{g \in \Gamma}$, where $\xi_g(x) = 1_{\{g^{-1}x\}}$ and $\Gamma$ is a countable subgroup of $[\R]$ which generates $\R$ (which exists by \cite{FM75a}, Thm. 1). The \emph{regular representation} of $\R$ is then the representation $\lambda$ on $X \ast \H$ given by $\lambda(x, y) = \id$ for all $(x, y) \in \R$.

Given representations $\pi$ on $X \ast \H$ and $\rho$ on $X \ast \mc{K}$, we say that $\pi$ and $\rho$ are \emph{unitarily equivalent} if there is a family of unitaries $\{U_x \in \U(\H_x, \K_x)\}_{x \in X}$ with 
\[U_x\pi(x, y) = \rho(x, y)U_y \quad \text{for all} \quad (x, y) \in \R,\] 
and such that $x \mapsto U_x\xi(x)$ is in $S(X \ast \mc{K})$ for each $\xi \in S(X \ast \H)$. 
We say that $\pi$ is \emph{weakly contained} in $\rho$, written $\pi \prec \rho$, if for any $\epsilon > 0$, $\xi \in S(X \ast \H)$, and $E \subset \R$ with $m(E) < \infty$, there exists $\{\eta_1, \dots, \eta_m\} \subset S(X \ast \K)$ with
\[m(\{(x, y) \in E : |\< \pi(x, y)\xi(y), \xi(x) \> - \sum_{i = 1}^m \< \rho(x, y)\eta_i(y), \eta_i(x) \>| \ge \epsilon\}) < \epsilon  \]
A representation $\pi$ on $X \ast \H$ is called \emph{mixing} (cf. \cite{Ki14}, Def. 4.4) if 
for every $\epsilon, \delta > 0$ and $\xi, \eta \in S(X \ast \H)$ with $\|\xi(x)\| = \|\eta(x)\| = 1$ a.e., there is $E \subset X$ with $\mu(X \setminus E) < \delta$ such that 
\[ \left|\{y \in [x]_{\R|_{E}} : |\<\pi(x, y)\xi(y), \eta(x)\>| > \epsilon\}\right| < \infty  \quad \text{for $\mu$-a.e. $x \in E$}\]

A \emph{1-cocycle} for a representation $\pi$ on $X \ast \H$ is a map $(x, y) \mapsto b(x, y) \in \H_x$ on $\R$ such that for $\mu$-a.e. $x \in X$, 
\begin{equation}\label{E: cocycle} 
b(x, z) = b(x, y) + \pi(x, y)b(y, z) \quad \text{for all} \quad (x, y), (y, z) \in \R,
\end{equation}
and such that $(x, y) \mapsto (x, b(x, y)) \in X \ast \H$ is measurable. 

A 1-cocycle $b$ is a \emph{coboundary} if there is a measurable section $\xi$ of $X \ast \H$ such that $b(x, y) = \xi(x) - \pi(x, y)\xi(y)$ for $m$-a.e. $(x, y) \in \R$, and a pair of 1-cocycles $b$ and $b'$ are \emph{cohomologous} if $b - b'$ is a coboundary.

We define a 1-cocycle to be \emph{bounded} if there exists a sequence of measurable subsets $\{E_n\}_{n = 1}^\infty$ of $X$ with $\mu(\bigcup_{n = 1}^\infty E_n) = 1$ and $\sup\{\|b(x, y)\|: (x, y) \in \R|_{E_n}\} < \infty$ for each $n \ge 1$. With this definition, the analysis of Anantharaman-Delaroche \cite{A-D03} gives the following equivalence:

\begin{lemma}
A 1-cocycle $b$ for a representation $\pi$ of $\R$ on $X \ast \H$ is a coboundary if and only if it is bounded.
\end{lemma}
\begin{proof}
Suppose there is $\xi \in S(X \ast \H)$ such that $b(x, y) = \xi(x) - \pi(x, y)\xi(y)$ for $m$-a.e. $(x, y) \in \R$. Then for $n \ge 1$ set 
$E_n = \{x \in X: \|\xi(x)\| \le n\}$. Then $\bigcup_{n = 1}^\infty E_n = X$ and for $(x, y) \in \R|_{E_n}$ we have $\|b(x, y)\| \le \|\xi(x)\| + \|\pi(x, y)\xi(y)\| \le 2n < \infty$. 

Conversely, consider a sequence of measurable subsets $\{E_n\}_{n = 1}^\infty$ of $X$ with $\mu(\bigcup_{n = 1}^\infty E_n) = 1$ and $\sup\{\|b(x, y)\|: (x, y) \in \R|_{E_n}\} < \infty$ for each $n \ge 1$. Then by Lemma 3.21 of \cite{A-D03}, for each $n$ we know that $b$ restricted to $\R|_{E_n}$ is a coboundary, i.e., there is $\xi_n \in S(E_n \ast \H)$ with $b(x, y) = \xi_n(x) - \pi(x, y)\xi_n(y)$ for $m$-a.e. $(x, y) \in \R|_{E_n}$. 
We can then extend $\xi_n$ to the $\R$-saturation $F_n = \bigcup_{x \in E_n} [x]_\R$ by the formula 
$\xi_n(x) = b(x, y) + \pi(x, y)\xi_n(y)$ for some $y \in E_n$ such that $(x, y) \in \R$. This definition does not depend on the choice of $y$; if $z \in E_n$ with $(x, z) \in \R$, then 
\begin{align*}
[b(x, y) + \pi(x, y)\xi_n(y)] -  [b(x, z) + \pi(x, z)\xi_n(z)] 
&= [b(x, y) - b(x, z)] + \pi(x, y)\xi_n(y) -  \pi(x, z)\xi_n(z) \\
 &= -\pi(x, y)b(y, z) + \pi(x, y)\xi_n(y) -  \pi(x, z)\xi_n(z)\\
 &= \pi(x, y)[-b(y, z) + \xi_n(y) - \pi(y, z)\xi_n(z)] \\
 &= \pi(z, y)[0] = 0. 
\end{align*}

Thus we have a sequence $\{\xi_n\}_{n = 1}^\infty$ such that $\xi_n \in S(F_n \ast \H)$, $b(x, y) = \xi_n(x) - \pi(x, y)\xi_n(y)$ for $m$-a.e. $(x, y) \in \R|_{F_n}$. Now for $x \in F = \bigcup_{n = 1}^\infty F_n$, define 
\begin{align}
\xi(x) = \xi_{n_x}(x), \quad \text{where} \quad n_x = \min \{n \ge 1 : x \in F_n\},
\end{align}
and let $\xi(x) = 0$ for $x \notin F$. 
Note that if $(x, y) \in \R$, then $n_x = n_y$ since each set $F_n$ is $\R$-invariant.
Thus
\[b(x, y) = \xi_{n_x}(x) - \pi(x, y)\xi_{n_y}(y) = \xi(x) - \pi(x, y)\xi(y)\]
for $(x, y) \in \R|_F$. Since $\mu(F) = 1$, this is $m$-a.e. $(x, y) \in \R$. 

Finally, to see that $\xi$ is measurable, note that $\xi(x) = \xi_n(x)$ for all $x \in F_n \setminus \bigcup_{k = 1}^{n-1} F_k$ which (along with $F^c$) decompose $X$ into measurable subsets on which the restriction of $\xi$ is measurable. 
\end{proof}

Thus a 1-cocycle that is not a coboundary must be \emph{unbounded} (i.e. not bounded), for which we have another useful characterization.

\begin{lemma}
A 1-cocycle $b$ for a representation $\pi$ of $\R$ on $X \ast \H$ is unbounded if and only if there is $\delta > 0$ such that for any $R > 0$ there is $g \in [\R]$ with $\mu(\{x \in X: \|b(x, g^{-1}x)\| > R\}) \ge \delta$. 
\end{lemma}
\begin{proof}
First, suppose there is $\delta > 0$ such that for any $R > 0$ there is $g \in [\R]$ with $\mu(\{\|b(x, g^{-1}x)\| \ge R\}) \ge \delta$. 
Then if $b$ were bounded, there would be measurable $E \subset X$ such that $\mu(E) > 1- \frac{\delta}{2}$ and $R = \sup \{\|b(x, y)\| : (x, y) \in \R|_{E}\} < \infty$. But then there would be $g \in [\R]$ such that $F = \{\|b(x, g^{-1}x)\| > R\}$ has $\mu(F) \ge \delta$. Noting that $g^{-1}(E \cap F) \subset E^c$, we then would have
\begin{align*}
\mu(E \cup F) &= \mu(E) + \mu(F) - \mu(E \cap F) \ge \mu(E) + \mu(F) - \mu(E^c) = 2\mu(E) + \mu(F) - 1 \\
&> 2(1 - \frac{\delta}{2}) + \delta - 1 = 1,
\end{align*}
which is impossible. 

Conversely suppose that $b$ is unbounded. By Feldman and Moore \cite{FM75a}, $\R = \R(\Gamma \curvearrowright X)$ for a pmp action of a countable group $\Gamma$. Moreover, $\Gamma$ can be chosen such that $(x, y) \in \R$ if and only if $y = hx$ for some $h \in \Gamma$ with $h^2 = e$. Since $\Gamma$ is countable, enumerate the elements of $\Gamma$ of order $\le 2$ as $\{h_n\}_{n = 1}^\infty$. 

For each $n \ge 1$, we recursively define a sequence of measurable subsets $\{A^n_k\}_{k = 1}^\infty$. 
Let $A^n_1 = \{x \in X: \|b(x, h_1x)\| \ge n\}$ and given $A^n_1, \dots, A^n_{k-1}$,
define $F^n_k = X \setminus \left[\bigcup_{j = 1}^{k-1} A^n_j\right]$ and 
\begin{align}
A^n_k = \{x \in F^n_k \cap h_kF^n_k: \|b(x, h_kx)\| \ge n\}
\end{align}
Set $A^n = \bigsqcup_{k = 1}^\infty A^n_k$.
Note that 
\[\|b(h_kx, h_k^2x)\| = \|b(h_kx, x)\| = \|-\pi(h_kx, x)b(x, h_kx)\| = \|b(x, h_kx)\|\]
from which it follows that $h_kA^n_k = A^n_k$ for every $k$. We can therefore define $g_n \in [\R]$ by the formula
\begin{align}
g_nx = 
\begin{cases}
h_kx, &\text{if }x \in A^n_k; \\
x, &\text{if }x \notin A^n.
\end{cases}
\end{align}
Note that $g_n^2 = e$ and $\|b(x, g_nx)\| \ge n$ for all $x \in A^n$.

Now set $E_n = X \setminus A^n$. For $x \in E_n = \bigcap_{k = 1}^\infty F^n_k$ we have $\|b(x, h_kx)\| < n$ for all $k$ such that $x \in h_kF^n_k$. Hence 
\begin{align*}
n &\ge \sup \{\|b(x, h_kx)\| : k \ge 1, x \in E_n \cap h_kF^n_k\} \\
&\ge \sup \{\|b(x, h_kx)\| : k \ge 1, x \in E_n \cap h_kE_n\} \\
&= \sup \{\|b(x, y)\| : (x, y) \in \R|_{E_n}\} 
\end{align*}
Therefore setting $\delta = 1 - \mu(\bigcup_{n = 1}^\infty E_n)$, we must have $\delta > 0$, since otherwise $b$ would be bounded. For any $R > 0$, taking some integer $n > R$ we have 
\[
\mu(\{\|b(x, g_n^{-1}x)\| \ge R\}) \ge \mu(\{\|b(x, g_n^{-1}x)\| \ge n\}) \ge \mu(A^n) = 1 - \mu(E_n) \ge \delta.
\]

\end{proof}

\subsection{Orbit Equivalence Relations}\label{SS: orbit equivalence relations}

Given a countable group $\Gamma$ with a pmp action $\Gamma \on (X, \mu)$, the \emph{orbit equivalence relation} $\R(\Gamma \on X)$ is defined by
\[
(x, y) \in \R(\Gamma \on X) \quad \iff \quad y = gx \text{ for some } g \in \Gamma,
\]
and two group actions are \emph{orbit equivalent (OE)} if and only if they have isomorphic orbit equivalence relations. 

In the case where $\R = \R(\Gamma \on X)$ for a free\footnote{
$\Gamma \on (X, \mu)$ is \emph{free} if $\mu(\{x \in X: gx = x\}) = 0$ for each nonidentity $g \in \Gamma$.
} pmp action of $\Gamma$, then $L(\R) \cong L^\infty(X) \rtimes \Gamma$, and for this reason the algebra $L(\R)$ is sometimes called the \emph{generalized group-measure space von Neumann algebra}.
 If $\Gamma$ is an amenable group then $\R(\Gamma \on X)$ is amenable, and the converse holds if the action is free.
Feldman and Moore showed in \cite{FM75a} that any countable pmp $\R$ arises from the action of a countable group, however this action cannot always be taken to be free, a question which was settled by Furman in \cite{Fu99b}.

If $\R = \R(\Gamma \on X)$ for the free pmp action of a countable group $\Gamma$, then any group representation $\pi: \Gamma \to \U(\H)$ of $\Gamma$ on a Hilbert space $\H$ gives rise to a representation $\pi_\R$ of $\R$, and a 1-cocycle $b$ for $\pi$ gives a 1-cocycle $b_\R$ for $\pi_\R$ as follows. 
We represent $\R$ on the Hilbert bundle $X \ast \K$ where $\K_x = \H$ for all $x \in X$.
Let $E_0 = \{x \in X: gx = x \text{ for some nonidentity } g \in \Gamma\}$. Then $\mu(E_0) = 0$ since $\Gamma$ is countable and the action is free. 
Define
\begin{align} \label{E: group2R}
\begin{split}
\pi_\R(x, g^{-1}x) = \pi(g), \quad &\text{and} \\
b_\R(x, g^{-1}x) = b(g),  \quad &\text{for} \quad g \in \Gamma, x \notin E_0,
\end{split}
\end{align}
and since $\mu(E_0) = 0$, for $x \in E_0$ take (say) $\pi(x, y) = \id$ and $b(x, y) = 0$. One can check that $\pi_\R$ is mixing if $\pi$ is mixing and $b_\R$ is unbounded if $b$ is unbounded. Moreover if $\pi \prec \rho$ for another representation $\rho$ of $\Gamma$, then $\pi_\R \prec \rho_\R$ as well. When $\pi$ is either the left or right regular representation, then $\pi_\R$ is unitarily equivalent to the regular representation $\lambda$. 

\subsection{Relative Mixingness and Weak Containment of Bimodules}\label{SS: bimodules}

We recall a few useful notions for bimodules over von Neumann algebras. 
Let $N \subset M$ be a von Neumann subalgebra. An $M$-$M$ bimodule $\MHM$ is \emph{mixing relative to $N$} if for any sequence $\{x_n\} \subset (M)_1$ with $\|\E_N(yx_nz)\|_2 \to 0$ for all $y, z \in M$, we have 
\begin{align}\label{E: Mixrel}
\lim_{n \to \infty} \sup_{y \in (M)_1} |\<x_n \xi y, \eta \>| = \lim_{n \to \infty} \sup_{y \in (M)_1} |\<y \xi x_n, \eta \>| = 0 \quad \text{for all} \quad \xi, \eta \in \H.
\end{align}

An $M$-$N$ bimodule $\MHN$ is \emph{weakly contained} in a $M$-$N$ bimodule $\MKN$, written $\MHN \prec \MKN$, if for any $\epsilon > 0$, finite subsets $F_1 \subset M, F_2 \subset N$, and $\xi \in \H$, there are $\eta_1, \dots, \eta_n \in \K$ such that 
\begin{align}\label{E: weakcontainment}
|\<x\xi y, \xi\> - \sum_{j = 1}^n \<x\eta_jy, \eta_j\>| < \epsilon \quad \text{for all} \quad x \in F_1, y \in F_2.
\end{align}

Given bimodules $\MHN$ and $\NKP$, we can form Connes' fusion $M$-$P$ bimodule ${_M \H \sten[N] \K_P}$ which satisfies $\xi a \sten_N \eta = \xi \sten_N a\eta$ for all $a \in N$, $\xi \in \H$, $\eta \in \K$ (see \cite{PV11} for a construction). If $\MHN \prec \MKN$, then ${_M \H \sten[N] \mc{L}_P} \prec {_M \K \sten[N] \mc{L}_P}$ for any $N$-$P$ bimodule $\mc{L}$,
and ${_P \mc{L} \sten[M] \H_N} \prec {_P \mc{L} \sten[M] \K_N}$ for any $P$-$M$ bimodule $\mc{L}$.

The following lemma is standard and appears in Remark 3.7 of \cite{Va10b}, for instance. We include the proof below for completeness.
\begin{lemma} \label{L: common}
Let $Q \subset M$ and let $\H$ be an $M$-$M$ bimodule. Suppose that $\{\xi_n\}_1^\infty \subset \H$, $\epsilon, \kappa > 0$ are such that 

(i) $\|\xi_n\| \ge \epsilon$ for all $n$, 

(ii) $\|x\xi_n\| \le \kappa\|x\|_2$ for $x \in M$ and all $n$, and 

(iii) $\|x\xi_n - \xi_nx\| \to 0$ for each $x \in Q$.
\\
Then there is a nonzero projection $z \in \Z(Q' \cap M)$ such that $_{M}[L^2(M)z]_{Qz} \prec \; _{M}[\H z]_{Qz}$. 
\end{lemma}
\begin{proof}
For each $x \in M$ and $n \ge 1$ set $\phi_n(x) = \<x\xi_n, \xi_n\>$. By \emph{(ii)} we have $0 \le \phi_n \le \kappa^2\tau$ so there is $T_n \in M$ with $0 < T_n \le \kappa^2$ such that $\<x\xi_n, \xi_n\> = \tau(xT_n)$ for all $x \in M$. Since $T_n \le \kappa^2$ for all $n$, passing to a subsequence we may assume that $T_n \to T$ weakly for some $T \in M$ with $0 \le T \le \kappa^2$. Moreover, $T \ne 0$ since $\tau(T_n) = \|\xi_n\|^2 \ge \epsilon^2$ for all $n$, and by \emph{(iii)} we have $T \in Q' \cap M$. Let $\delta > 0$ be small enough that $p = 1_{(\delta, \kappa^2]}(T)$ is nonzero. Then set $S = f(T)$ where $f(t) = ({{1_{(\delta, \kappa^2]}(t)}/{t}})^{1/2}$ for $t \in \sigma(T)$ so that $S^2T = p$. Since $p \in Q' \cap M$ is nonzero, there is $p' \in Q' \cap M$ with $p' \le p$ and $\E_{\Z(Q' \cap M)}(p') = \frac{1}{m}z$ for some $m \in \ZZ_{> 0}$ and $z \in \Z(Q' \cap M)$.
Let $v_1, v_2, \dots, v_m \in Q' \cap M$ be partial isometries with $v_j^*v_j = p'$ for all $1 \le j \le m$ and $\sum_{j = 1}^m v_jv_j^* = z$. Set $\eta_n^j = v_jS'\xi_n$ and $S' = Sp'$ so that $S'^2T = p'$. Then for any $x, a \in M$ and $y \in Qz$ we have 
 \begin{align*}
 \<xay, a\> &= \tau(a^*xay) 
 = \sum_{j = 1}^m \tau(a^*xayv_jv_j^*) 
 = \sum_{j = 1}^m \tau(xyv_jS'TS'v_j^*) 
 = \lim_{n \to \infty}  \sum_{j = 1}^m \tau(S'v_j^*xyv_jS'T_n) \\
 &= \lim_{n \to \infty} \sum_{j = 1}^m \<S'v_j^*xyv_jS'\xi_n, \xi_n\>
 = \lim_{n \to \infty} \sum_{j = 1}^m \<x\eta_n^jy, \eta_n^j\> 
 \end{align*}
and thus $_{M}[L^2(M)z]_{Qz} \prec \; _{M}[\H z]_{Qz}$.
\end{proof}

\subsection{Relative Amenability}\label{SS: Relatively amenable subalgebras}

The notion of relative amenability for von Neumann subalgebras is due to Ozawa and Popa in \cite{OP07}, from which we get the following:

\begin{theorem}[\cite{OP07}]
Let $N, Q$ be von Neumann subalgebras of $(M, \tau)$ which contain $1_M$. Then the following are equivalent:

(1). $N$ is amenable relative to $Q$ inside $M$;

(2). There is an $N$-central state $\phi$ on $\<M, e_Q\>$ such that $\phi|_M = \tau$;

(3). There is a conditional expectation $\Phi: \<M, e_Q\> \to N$ such that $\Phi|_M = \E_N$;

(4). There is $\{\xi_n\} \subset L^2(M) \sten[Q] L^2(M)$ such that for each $x \in M$, $b \in N$ we have $\<x\xi_n, \xi_n\> \to \tau(x)$ and $\|b\xi_n - \xi_nb\| \to 0$ as $n \to \infty$; 

(5). ${}_ML^2(M)_N \prec {}_ML^2(M) \sten[Q] L^2(M)_N$. 
\end{theorem}

This generalizes the notion of amenability for subalgebras: $N$ is amenable iff it is amenable relative to $\CC$ inside $M$ for some (and hence all) $M \supset N$. 

We will need the following useful proposition due to Popa and Vaes:

\begin{prop}[Proposition 2.7 of \cite{PV11}] \label{T: commuting square}
Let $N, Q_1, Q_2$ be von Neumann subalgebras of $(M, \tau)$ which contain $1_M$. Suppose that $M = \N_M(Q_1)''$ and $[e_{Q_1}, e_{Q_2}] = 0$, and that $N$ is amenable relative to $Q_i$ for $i = 1, 2$. Then $N$ is amenable relative to $Q_1 \cap Q_2$. 
\end{prop}

\subsection{Amplifications}\label{SS: amplifications}

In order to give the prime factorization result in Section \ref{S: factor} and the application to measure equivalence in Section \ref{S: ME}, we will need the language of amplifications.  
For a $\rm{II}_1$ factor $(M, \tau)$, we consider the type ${\rm II}_\infty$ factor $M^\infty = M \ten B(\ell^2(\ZZ))$. If we denote by $\Tr$ the semifinite trace on $B(\ell^2(\ZZ))$, then $\tau \sten \Tr$ gives a semifinite trace on $M^\infty$.
For any $t > 0$, the \emph{amplification} of $M$ by $t$ is the $\rm{II}_1$ factor $M^t = PM^\infty P$ for a projection $P \in M^\infty$ satisfying $(\tau \sten \Tr)(P) = t$. Note that such a projection exists since $M$ is ${\rm II}_1$ and that $M^t$ is well defined up to unitary conjugacy in $M^\infty$. If $M$ is type ${\rm I}_n$ for some $n \in \ZZ_{>0}$ such $P$ exists provided $nt \in \ZZ$ and in this case we define $M^t$ as above. For $s, t > 0$, $(M^s)^t = M^{st}$. 
 
Now consider the tracial factor $L(\R)$ for ergodic $\R$. If $\R$ is infinite, ergodicity implies that the space $(X, \mu)$ must be non-atomic and $L(\R)$ is a type ${\rm II}_1$ factor. For such $\R$ and $t > 0$, we can define as follows the \emph{amplification} $\R^t$ of $\R$ in such a way that $L(\R^t) \cong L(\R)^t$.

Consider the measure space $(X^\infty, \mu \sten \#) = (X \times \ZZ, \mu \sten \#)$, where $\#$ denotes the counting measure on $\ZZ$. Then $\R^\infty = \R \times \ZZ^2 \subset X^\infty \times X^\infty$ is a countable measurable equivalence relation. 
For $t > 0$, define $\R^t = \R^\infty|_E$ for measurable $E \subset X^\infty$ with $(\mu \sten \#)(E) = t$. Such a set $E$ exists since $X$ is non-atomic, and using the ergodicity of $\R$ one can show that $\R^t$ is well defined up to isomorphism. From this it further follows that $(\R^t)^s = \R^{ts}$ for $t, s > 0$.

If $\R$ has a representation $\pi$ with 1-cocycle $b$ on a Hilbert bundle $X \ast \H$, we can form the Hilbert bundle $E \ast \H^t$, where $E$ is as above and where $\H^t_{(x, k)} = \H_x$ for each $(x, k) \in E$. Then we can define a representation $\pi^t$ of $\R^t$ with 1-cocycle $b^t$ by
\begin{align} \label{E: pib^t}
\begin{split}
\pi^t((x, k), (y, m)) = \pi(x, y), \quad &\text{and} \\
b^t((x, k), (y, m)) = b(x, y), \quad &\text{for} \quad ((x, k), (y, m)) \in \R^t. 
\end{split}
\end{align}
For any $t > 0$, $\pi$ is mixing if and only if $\pi^t$ is mixing, $b$ is unbounded if and only if $b^t$ is unbounded, and for another representation $\rho$ of $\R$, $\pi \prec \rho$ if and only if $\pi^t \prec \rho^t$.  

\subsection{Popa's intertwining by bimodules}\label{SS: twine}

We will make essential use of the following theorem of Popa, fundamental to deformation/rigidity theory:

\begin{theorem}[Popa's Intertwining by Bimodules, Theorem 2.1 of \cite{Po03}]\label{T: twine}
Let $N$ and $P$ be unital subalgebras of a tracial von Neumann algebra $M$. The following are equivalent:

(1) There is no sequence $\{u_n\} \subset \U(N)$ such that $\|\E_P(xu_ny)\|_2 \to 0$ as $n \to \infty$ for every $x, y \in M$; 

(2) There is a $N$-$P$ submodule $\H$ of $L^2(M)$ with ${\rm dim}_P(\H) < \infty$;

(3) There are nonzero projections $p \in N$, $f \in P$, a unital normal $*$-homomorphism $\theta:~pNp\to fPf$, and a partial isometry $v \in M$ such that
\begin{align*}   
\theta(x)v = vx \quad \text{for all} \quad x \in pNp, \quad 
v^*v \in (N' \cap M)p, \quad \text{and} \quad
vv^* \in \theta(pNp)' \cap fMf.   
\end{align*}
\end{theorem}
If the above equivalent conditions hold, we say that \emph{$N$ intertwines into $P$ inside $M$}, written $N \prec_M P$, or simply $N \prec P$ when there is no danger of confusion.

\section{Deducing Primeness from an $s$-Malleable Deformation} \label{S: deduce}

In this section, we review how an $s$-malleable deformation can be used to prove primeness results using a technique introduced by Popa in \cite{Po06c}. We define an \emph{s-malleable deformation} of a tracial von Neumann algebra $M$ as an inclusion $M \subset \tilde M$ into some tracial $\tilde M$, along with a continuous action $\alpha: \RR \to \Aut(\tilde M)$, and $\beta \in \Aut(\tilde M)$ such that $\beta|_{M} = \id$, $\beta^2 = \id$, and $\beta \circ \alpha_t = \alpha_{-t} \circ \beta$ for all $t \in \RR$. 

To exploit an $s$-malleable deformation, we will use Popa's transversality inequality from \cite{Po06a}, part (1) of the following lemma. We include as part (2) another well known inequality as we shall use the pair several times in combination. 

\begin{lemma} \label{L: trans}
Let $\alpha: \RR \to \Aut(\tilde M)$, $\beta \in \Aut(\tilde M)$ be an $s$-malleable deformation of $M \subset \tilde M$. Set $\delta_t(x) = \alpha_t(x) - \E_M(\alpha_t(x))$ for $x \in M$. Then for all $x, y \in M$ and $t \in \RR$,
\begin{flalign*}
&\quad(1). \quad\|\delta_{2t}(x)\|_2 \le 2\|\alpha_{2t}(x) - x\|_2 \le 4\|\delta_t(x)\|_2, \text{ and} &\\ 
&\quad(2). \quad\|[\delta_t(x), y]\|_2 \le 2\|x\|\|\alpha_t(y) - y\|_2 + \|[x, y]\|_2. &
\end{flalign*}
\end{lemma}
\begin{proof} For (1), 
\begin{align*}
\|\delta_t(x)\|_2 
&\le \|\alpha_t(x) - x\|_2 + \|x - \E_M(\alpha_t(x))\|_2 
= \|\alpha_t(x) - x\|_2 + \|\E_M(x - \alpha_t(x))\|_2 \\
&\le 2\|\alpha_t(x) - x\|_2,
\end{align*}
and since $\beta \alpha_t = \alpha_{-t}\beta$ and $\beta|_M = \id$, we have
\begin{align*}
\|\alpha_{2t}(x) - x\|_2 
&= \|\alpha_t(x) - \alpha_{-t}(x)\|_2 
\le \|\alpha_t(x) - \E_M(\alpha_t(x))\|_2 + \|\alpha_{-t}(x) - \E_M(\alpha_t(x))\|_2 \\
&= \|\delta_t(x)\|_2 + \|\beta(\alpha_t(x) - \E_M(\alpha_t(x)))\|_2
= 2\|\delta_t(x)\|_2.
\end{align*}
For (2), 
\begin{align*} 
\|[\delta_t(x), y]\|_2 
& = \|(1-E_M)([\alpha_t(x), y])\|_2 
\le \|[\alpha_t(x), y]\|_2 \\
&\le \|\alpha_t(x)y - \alpha_t(x)\alpha_t(y)\|_2 + \|[\alpha_t(x), \alpha_t(y)]\|_2 + \|\alpha_t(y)\alpha_t(x) - y\alpha_t(x)\|_2  \\
&\le 2\|x\|\|\alpha_t(y) - y\|_2 + \|[x, y]\|_2 \qedhere \\
\end{align*}
\end{proof}

We can now show how an $s$-malleable deformation can be used to prove primeness.

\begin{theorem}[Popa's Spectral Gap Argument] \label{T: deduce}
Let $M$ be a tracial von Neumann algebra with no amenable direct summand which admits an $s$-malleable deformation $\{\alpha_t\}_{t \in \RR} \subset \emph{Aut} (\tilde M)$ for some tracial von Neumann algebra $\tilde M \supset M$. Suppose that the $M$-$M$ bimodule $_ML^2(\tilde M) \ominus L^2(M)_M$ is weakly contained in the coarse $M$-$M$ bimodule and mixing relative to some abelian subalgebra $A \subset M$. 
Then there is a central projection $z \in \Z(M)$ such that 
\begin{enumerate}
\item $\alpha_t \to \id$ uniformly in $\|\cdot\|_2$ on the unit ball $(Mz)_1$ as $t \to 0$, and
\item $M(1-z)$ is prime.
\end{enumerate}

In particular, if the convergence $\alpha_t \to \id$ as $t \to 0$ is not uniform, then $M \ncong N \ten Q$ for any $N$ and $Q$ of type ${\rm II}$.
\end{theorem} 
\begin{proof} 
Using Zorn's Lemma, let $z \in \Z(M)$ denote the maximal central projection such that $\alpha_t \to \id$ uniformly in $\|\cdot\|_2$ on the unit ball $(Mz)_1$. Then (1) is satisfied and for any central projection $z' \le 1-z$ we have $\alpha_t \to \id$ non-uniformly in $\|\cdot\|_2$ on $(Mz')_1$. 
Now suppose toward a contradiction that $M(1-z) = N \ten Q$ with $N$ and $Q$ not of type ${\rm I}$. Since $M$ has no amenable direct summand, we assume without loss of generality that $Q$ also has no amenable direct summand. As previously, set $\delta_t(x) = \alpha_t(x) - \E_M(\alpha_t(x))$ for $x \in M$.
 
First suppose that $\alpha_t \to \id$ is not uniform in $\|\cdot\|_2$ on $(N)_1$. Then by part (1) of Lemma \ref{L: trans}, $\delta_t \to 0$ is not uniform in $\|\cdot\|_2$ on $(N)_1$. Hence there is $\epsilon > 0$ and sequences $\{a_n\} \in (N)_1$, $\{t_n\} \subset \RR$, with $t_n \to 0$ and $\|\delta_{t_n}(a_n)\|_2 > \epsilon$ for all $n$. 
For $x \in Q$, we have $[x, a_n] = 0$ and $\|\alpha_{t_n}(x) - x\|_2 \to 0$ as $n \to \infty$, so part (2) of Lemma \ref{L: trans} gives $\|[\delta_{t_n}(a_n), x]\|_2 \to 0$. 
We also have $\|x\delta_{t_n}(a_n)\|_2 \le \|x\|_2$, so applying Lemma \ref{L: common} with $\H = L^2(\tilde M) \ominus L^2(M)$, and using our assumption that $_ML^2(\tilde M) \ominus L^2(M)_M$ is weakly contained in the coarse $M$-$M$ bimodule, there is a projection $q \in \Z(Q' \cap Q) = \Z(Q)$ such that 
\[_{Q}L^2(Qq)_{Qq}  \prec  \;_{Q}[(L^2(\tilde M) \ominus L^2(M))q]_{Qq} \prec\; _{Q}[L^2(M) \ten L^2(M)q]_{Qq} \prec\; _{Q}L^2(Q) \ten L^2(Qq)_{Qq}\] 
and hence ${}_{Qq}L^2(Qq)_{Qq}  \prec {}_{Qq}L^2(Qq) \ten L^2(Qq)_{Qq}$, which contradicts the fact that $Q$ has no amenable direct summand. Thus we must have $\alpha_t \to \id$ uniformly in $\|\cdot\|_2$ on $(N)_1$. 

Next, since $N$ is not type ${\rm I}$, there is $z' \in \Z(N) \subset \Z(M)(1-z)$ such that $Nz'$ is type ${\rm II}$. Then since $A$ is abelian and $Nz'$ is type ${\rm II}$, we have $Nz' \nprec_M A$, so it follows from Theorem \ref{T: twine} that there is a sequence $\{u_n\} \subset \U(Nz')$ such that for each $x, y \in M$, $\|\E_A(xu_ny)\|_2 \to 0$ as $n \to \infty$. 
Since $_ML^2(\tilde M) \ominus L^2(M)_M$ is mixing relative to $A$, we have that $\<u_n\delta_t(x), \delta_t(x)u_n\> \to 0$ as $n \to \infty$ for all $x \in M$. 
Note that for any $t \in \RR$, $x \in (Qz')_1$ we have 
\begin{align*}
\|\delta_t(x) - z'\delta_t(x)\|_2 = \|(1 - \E_M)(\alpha_t(z'x) - z'\alpha_t(x))\|_2 \le \|\alpha_t(z') - z'\|_2
\end{align*}
and so using both parts of Lemma \ref{L: trans}, we have
\begin{align} \label{E: usemix}
\begin{split}
\|\alpha_{2t}(x) - x\|_2 
&\le 2\|\delta_t(x)\|_2 
\le 2\|\alpha_t(z') - z'\|_2 + 2\|z'\delta_t(x)\|_2 \\
&= 2\|\alpha_t(z') - z'\|_2 + \liminf_{n \to \infty} \left[2\|[u_n,\delta_t(x)]\|_2^2+4\text{Re}\<u_n\delta_t(x), \delta_t(x)u_n\>\right]^{1/2} \\
&\le 2\|\alpha_t(z') - z'\|_2 + \liminf_{n \to \infty} \left[\sqrt{8}\|\alpha_t(u_n) - u_n\|_2\right] \\
&\le (2+\sqrt{8})\sup_{a \in (N)_1} \|\alpha_t(a) - a\|_2 \longrightarrow 0 \text{ as } t \to 0.
\end{split}
\end{align}
As this convergence is independent of $x$, this shows that $\alpha_t \to \id$ uniformly in $\|\cdot\|_2$ on $(Qz')_1$.

Now fix any $\epsilon > 0$, and
let $t_0 > 0$ be such that for $|t| < t_0$ we have $\|\alpha_t(x) - x\|_2 < \frac{\epsilon}{7}$ for all $x \in (N)_1 \cup (Qz')_1$. 
Then for $u \in \U(N), v \in \U(Qz')$ we have
\begin{align*}
\|\alpha_t(u)\alpha_t(v)v^*u^* - z'\|_2 
&\le \|\alpha_t(u)\alpha_t(v)z' - \alpha_t(u)\alpha_t(v)\|_2 + \|\alpha_t(u)\alpha_t(v) - uv\|_2 \\
&\le \|z' - \alpha_t(z')\|_2 + \|\alpha_t(u) - u\|_2 + \|\alpha_t(v) - v\|_2 < \frac{3\epsilon}{7}
\end{align*}
and so for $|t| < t_0$, the $\|\cdot\|_2$-closed convex hull $K_t$ of the set $\{\alpha_t(u)\alpha_t(v)v^*u^* : u \in \U(N), v \in \U(Qz')\}$ 
has $\|k - z'\|_2 \le \frac{3\epsilon}{7}$ for all $k \in K_t$. 
In particular, the unique element $k_t \in K_t$ of minimal $\|\cdot\|_2$ has $\|k_t-z'\|_2 \le \frac{3\epsilon}{7}$. 
Since $k_t$ is unique and $\alpha_t(u)K_t u^* = K_t$ for all $u \in \U(N) \cup \U(Qz')$, it follows that $\alpha_t(u)k_t u^* = k_t$ for all $u \in \U(N) \cup \U(Qz')$, and hence $\alpha_t(a)k_t = k_ta$ for all $a \in N \cup Qz'$. Then for any $a \in N$, $b \in Qz'$, we have 
\[\alpha_t(ab)k_t = \alpha_t(a)\alpha_t(b)k_t = \alpha_t(a) k_t b = k_t a b,\]
and $Mz' = M(1-z)z' = (N \ten Q)z'$, so in fact
$\alpha_t(x)k_t = k_tx$ for all $x \in Mz'$ and $|t| < t_0$.  
Thus for any $x \in (Mz')_1$ and $|t| < t_0$ we have 
\begin{align*}
\|\alpha_t(x)-x\|_2 &\le \|\alpha_t(z'x) - z'\alpha_t(x)\|_2 + \|z'\alpha_t(x) - \alpha_t(x)k_t\|_2 + \|k_tx - x\|_2 \\
&\le \|\alpha_t(z') - z'\|_2 + 2\|k_t - z'\|_2 
\le \frac{\epsilon}{7} + 2\cdot\frac{3\epsilon}{7} = \epsilon,
\end{align*}
which implies that $\alpha_t \to \id$ uniformly in $\|\cdot\|_2$ on $(Mz')_1$. But $z' \in \Z(M)$ with $z' \le 1-z$, so this is a contradiction and we conclude that $M(1-z)$ is indeed prime. 

In the particular case where the convergence $\alpha_t \to \id$ is not uniform in $\|\cdot\|_2$ on $(M)_1$, the above projection $z \in \Z(M)$ has $1 - z \ne 0$. Suppose toward a contradiction that $M \cong N \ten Q$ with $N$ and $Q$ of type ${\rm II}$.
Then since $M(1-z)$ is prime by the above result, the decomposition $M(1-z) = N(1-z) \ten Q(1-z)$ forces either $N(1-z)$ or $Q(1-z)$ to be type I. Assume without loss of generality that $N(1-z)$ is type ${\rm I}$. But then taking a nonzero abelian projection $p \in N(1-z)$, we would have a type ${\rm II}$ algebra $N$ intertwining in $M$ into an abelian algebra $Np \oplus \CC(1-p)$, which is impossible. Thus $M \ncong N \ten Q$ with $N$ and $Q$ of type ${\rm II}$.
\end{proof}

\vspace{.25cm}
\section{Gaussian Extension of $\R$ and $s$-Malleable Deformation of $L(\R)$} \label{S: deform}

In this section we construct the s-malleable deformation that will be used to prove the main result. 
In \cite{PS09} and \cite{Si10}, Peterson and Sinclair used 1-cocycles for group representations to build deformations; we follow this spirit in the setting of pmp equivalence relations.
To accomplish this, we generalize Bowen's \emph{Bernoulli shift extension of $\R$} (see \cite{Bo12}) to the \emph{Gaussian extension $\tilde \R$ of $\R$} arising from an orthogonal representation $\pi$ of $\R$. A 1-cocycle for $\pi$ will then give rise to the desired $s$-malleable deformation of $L(\R) \subset L(\tilde \R)$.  

\subsection{Gaussian extension of $\R$} \label{SS: gauss}
Let $\pi$ be an orthogonal representation of $\R$ on a real Hilbert bundle $X \ast \H$, and let $\{\xi_i\}_{i = 1}^\infty$ be an orthonormal fundamental sequence of sections for $X \ast \H$. Let
\[
(\Omega_x, \nu_x) = \prod_{i = 1}^{\dim \H_x} (\RR, \frac{1}{\sqrt{2\pi}}e^{-s^2/2}ds),
\]
and define
$\omega_x: \spn_\RR (\{\xi_{i}(x)\}_{i = 1}^{\dim \H_x}) \to \U(L^\infty(\Omega_x))$ by
\[\omega_x\left(\sum_{n = 1}^k a_n\xi_{i_n}(x)\right) = \exp({i\sqrt{2}\sum_{n = 1}^k a_nS^x_{i_n}})
\]
where $S^x_j$ is the coordinate function $S^x_j((s_i)_{i = 1}^{\dim \H_x}) = s_j$ for $j \le \dim \H_x$.

One can show that $\omega_x$ extends to a $\|\cdot\|_{\H_x} - \|\cdot\|_2$ continuous map $\omega_x : \H_x \to \U(L^\infty(\Omega_x))$ satisfying 
\begin{equation} \label{E: wprop}
\tau(\omega_x(\xi)) = e^{-\|\xi\|^2}, 
\quad \omega_x(\xi + \eta) = \omega_x(\xi)\omega_x(\eta), 
\quad \omega_x(\xi)^* = \omega_x(-\xi), 
\quad \text{for all $\xi, \eta \in \H_x$.}
\end{equation}

For $x \in X$, one can show that the linear span $D_x = \spn_\CC (\{\omega_x(\xi)\}_{\xi \in \H_x})$ has $D_x'' = \ol{D_x}^\wot = L^\infty(\Omega_x)$.
For $(x, y) \in \mc{R}$, define a $*$-homomorphism $\rho(x, y): D_y \to L^\infty(\Omega_x)$ by 
\[
\rho(x, y)\omega_y(\xi) = \omega_x(\pi(x, y)\xi),
\]
which is well defined and $\|\cdot\|_2$-isometric since \eqref{E: wprop} implies 
\[
\tau(\omega_y(\eta)^*\omega_y(\xi)) = \tau(\omega_x(\pi(x, y)\eta)^*\omega_x(\pi(x, y)\xi)) \quad \text{for all} \quad \xi, \eta \in \H_y.
\]
In particular, $\rho(x, y)$ is $\tau$-preserving, and so extends to a $*$-isomorphism $\rho(x, y): L^\infty(\Omega_y) \to L^\infty(\Omega_x)$.  
Let $\theta_{(x, y)}: \Omega_y \to \Omega_x$ be the induced measure space isomorphism such that $\rho(x, y)f = f \circ \theta_{(x, y)}^{-1}$ for all $f \in L^\infty(\Omega_y)$.

We now consider $X \ast \Omega$ as measurable bundle with $\sigma$-algebra generated by the maps $(x, r) \mapsto \omega_x(\sum_{i \in I}a_i\xi_i(x))(r)$ for $I \subset \NN$ finite and $a_i \in \RR$. 
A measure $\mu \ast \nu$ on $X \ast \Omega$ is then given by $[\mu \ast \nu](E) = \int_X \nu_x(E_x)d\mu(x)$, where $E_x = \{s \in \Omega_x: (x, s) \in E\}$. We define the \emph{Gaussian extension of $\R$} to be the equivalence relation $\tilde \R$ on $(X \ast \Omega, \mu \ast \nu)$ given by 
\begin{align}
((x, r), (y, s)) \in \tilde \R \iff (x, y) \in \R \text{ and }\theta_{(y, x)}(r) = s
\end{align}
leaving the reader to check that $\tilde \R$ is a countable pmp equivalence relation.

For $g \in [\R]$, we can define $\tilde g \in [\tilde \R]$ by $\tilde g(x, r) = (g x, \theta_{(g x, x)}(r))$. 
One can then check that the map $au_g \mapsto au_{\tilde g}$ for $a \in L^\infty (X)$, $g \in [\R]$ imbeds $L(\R)$ into $L(\tilde \R)$ and we identify $u_g$ and $u_{\tilde g}$ henceforth.
Moreover, noting that $\tilde \R = \bigcup_{g \in [\R]} \text{graph}(\tilde g)$, it follows that 
\begin{align}\label{E: L(tilde R)}
L(\tilde \R) = \vN{L^\infty(X \ast \Omega), \{u_{\tilde g} : g\in [\R]\}} = \vN{L^\infty(X \ast \Omega), L(\R)} \subset \B(L^2(\tilde \R))
\end{align}

\subsection{s-Malleable deformation of $L(\R)$}

Now consider a 1-cocycle $b$ for the representation $\pi$ on $X \ast \H$ above and let $M = L(\R)$ and $\tilde M = L(\tilde \R)$. Using the cocycle relation for $b$, one checks that 
\begin{align}\label{E: multiplicativecocycle}
c_t((x, r), (y, s)) = \omega_x(tb(x, y))(r)
\end{align} 
defines a one-parameter family $\{c_t\}_{t \in \RR}$ of multiplicative 1-cocyles of $\tilde \R$, and hence as in \eqref{E: auto}, a one-parameter family $\{\psi_{c_t}\}_{t \in \RR} \subset \Aut(\tilde M)$ which we will denote by $\{\alpha_t\}_{t \in \RR}$. Moreover, $c_{t_1}c_{t_2} = c_{t_1 + t_2}$ for all $t_1, t_2 \in \RR$, and hence $t \mapsto \alpha_t$ defines an action $\alpha: \RR \to \Aut(\tilde M)$. 
For any $a \in L^\infty(X \ast \Omega)$, $g \in [\R]$,
\begin{align*}
\|\alpha_t(au_g) - au_g\|_2^2 
&= \|f_{c_t, g}au_g-au_g\|_2^2 
\le \|a\|^2\|f_{c_t, g}-1\|_2^2 
= \|a\|^2\left[2 - 2\text{Re } \tau(f_{c_t, g})\right] \\
& = 2\|a\|^2\left[1 - \text{Re } \int_X \tau(\omega_x(tb(x, g^{-1}x))) d\mu(x)\right] \\
& = 2\|a\|^2\text{Re } \int_X \left[1 - e^{-t^2\|b(x, g^{-1}x)\|^2} \right] d\mu(x)
\to 0 \text{ as } t \to 0,
\end{align*}
where the convergence follows from Lebesgue's dominated convergence theorem. 
When combined with \eqref{E: L(tilde R)}, this shows that $\alpha: \RR \to \Aut(\tilde M)$ is a continuous action when $\Aut(\tilde M)$ is given the topology of pointwise $\|\cdot\|_2$ convergence. 

Next, one can check that defining $\beta_x(\omega_x(\xi)) = \omega_x(-\xi)$ for $x \in X$ gives rise to $\beta_x \in \Aut(L^\infty(\Omega_x))$, which leads to $\beta \in \Aut(L^\infty(X \ast \Omega))$ defined by $\beta(a)(x, r) = \beta_x(a(x, \cdot))(r)$ for $a \in L^\infty(X \ast \Omega)$. Then noting that $u_g\beta(a)u_g^* = \beta(u_gau_g^*)$ for all $a \in L^\infty(X \ast \Omega)$, $g \in [\tilde \R]$, one can check that $\beta$ extends to an $*$-automorphism of $\tilde M$ by the rule $\beta(au_g) = \beta(a)u_g$. We have $\beta^2 = \id$, $\beta|_M = \id$, and $\beta \circ \alpha_t = \alpha_{-t} \circ \beta$ since one can check that $\beta(f_{c_t, g}) = f_{c_{-t}, g}$ for each $g \in [\tilde \R]$. Hence $\alpha: \RR \to \Aut(\tilde M)$ is an $s$-malleable deformation of $M \subset \tilde M$. 

\section{Primeness of $L(\R)$} \label{S: main}

In this section, we prove the main result, Theorem \ref{T: main}. Before doing so, however, we pause to further analyze the maps $\rho(x, y): L^\infty(\Omega_y) \to L^\infty(\Omega_x)$ defined in Section \ref{SS: gauss}. Note first that each can be extended (and then restricted) to a unitary 
\[\rho(x, y): L^2(\Omega_y)\ominus \CC \to L^2(\Omega_x) \ominus \CC\]

Setting $\K_x = L^2(\Omega_x)\ominus \CC$ for $x \in X$, we now form the Hilbert bundle $X \ast \K$ with the $\sigma$-algebra determined by fundamental sections $\omega_0(\spn_\mathbb{Q} \{\xi_n\}_{n = 1}^\infty)$, where $\{\xi_n\}_{n = 1}^\infty$ is as in Section \ref{SS: gauss}, and $[\omega_0(\eta)](x) = \omega_x(\eta(x)) - e^{-\|\eta(x)\|^2}$ for $\eta \in S(X \ast \H)$. 
Noting that $\rho(x, y)\rho(y, z) = \rho(x, z)$ for all $(x, y), (y, z) \in \R$, we may then consider $\rho$ as a representation of $\R$ on $X \ast \K$. The following lemma makes explicit the relationship between $\rho$ and $\pi$.  

\begin{lemma} \label{L: chaos}
For each $x \in X$, let $\hat \H_x = \bigoplus_{n = 1}^\infty (\H_x \sten[\RR] \CC)^{\odot n}$.  The representation $\rho$ of $\R$ on $X \ast \K$ is unitarily equivalent to the representation $\hat \pi = \oplus_{n = 1}^\infty \pi_\CC^{\odot n}$ of $\R$ on $X \ast \hat\H$. 
\end{lemma}
\begin{proof}
For $x \in X$, define $U_x: D_x \to \CC \oplus \hat\H_x$ by 
$\omega_x(\xi) \mapsto e^{-\|\xi\|^2}\bigoplus_{n = 0}^\infty \frac{(i\sqrt{2}\xi)^{\odot n}}{n!}$ for $\xi \in H_x$, which is well defined and isometric since for any $\xi, \eta \in \H_x$, we have
\begin{align*}
& \left\<e^{-\|\xi\|^2}\bigoplus_{n = 0}^\infty \frac{(i\sqrt{2}\xi)^{\odot n}}{n!}, e^{-\|\eta\|^2}\bigoplus_{n = 0}^\infty \frac{(i\sqrt{2}\eta)^{\odot n}}{n!}\right\>
= e^{-\|\xi\|^2}e^{-\|\eta\|^2}\sum_{n = 0}^\infty \frac{2^n}{(n!)^2}\<\xi^{\odot n}, \eta^{\odot n}\> \\
=&\; e^{-\|\xi\|^2}e^{-\|\eta\|^2}\sum_{n = 0}^\infty \frac{2^n}{(n!)^2} n! (\<\xi, \eta\>)^n 
= e^{-\|\xi\|^2}e^{-\|\eta\|^2}e^{2\<\xi, \eta\>} 
= \tau(\omega_x(\eta)^*\omega_x(\xi))
\end{align*}
Certainly $\CC \subseteq U_x(D_x)$, and one can check that $\xi_1 \odot \cdots \odot \xi_n \in \ol{U_x(D_x)}$ for all $\xi_1, \dots, \xi_n \in \H_x$ by induction on $n$. Hence we extend this map to a unitary $U_x: L^2(\Omega_x) \to \CC \oplus \hat\H_x$. 

Then for $(x, y) \in \R$, it is immediate from the definitions of $\hat \pi$ and $\rho$ that 
\[[\id_\CC \oplus \hat \pi](x, y)U_y\omega_y(\xi) = U_x\omega_x(\pi(x, y)\xi) = U_x\rho(x, y)\omega_y(\xi)
\]
for all $\xi \in \H_y$, and hence
\[[\id_\CC \oplus \hat \pi] (x, y)U_ya = U_x\rho(x, y)a\]
 for all $a \in L^2(\Omega_y)$ since $L^2(\Omega_y) = \ol{\omega_y(\H_y)}$. 
In particular, since $U_y$ fixes $\CC$ for each $y \in X$, the lemma follows.   
\end{proof}

\subsection{$L(\R)$-$L(\R)$ bimodules arising from representations of $\R$}\label{SS: bimodule}

We will need one more tool before the proof of Theorem \ref{T: main}. 
Again let $M = L(\R)$ and $\tilde M = L(\tilde \R)$, and write $A$ for $L^\infty(X) \subset M$. 
Note that a representation $\pi$ of $\mc{R}$ on $X \ast \mc{H}$ induces a group representation $\tilde \pi: [\R] \to \mc{U}(\int_X^{\oplus} \H_x d\mu(x))$ defined by $\tilde\pi_g (\xi) (x) = \pi(x, g^{-1}x)\xi(g^{-1}x)$. Letting 
\begin{align}
\H^\pi := \left[\int_X^{\oplus} \H_x d\mu(x)\right] \sten[A] L^2(\R),
\end{align}
we would like to define an $M$-$M$ bimodule structure on $\H^\pi$. The intuition comes from the proof of the following analogue of Fell's absorption principle:
\begin{lemma}
Let $\pi$ be a representation of $\R$ on a measurable Hilbert bundle $X \ast \H$. Then $\pi \sten \lambda$ is unitarily equivalent to $\id_{\mc{S}} \sten \lambda$ for any orthonormal fundamental sequence of sections $\mc{S}$ for $X \ast \H$.
\end{lemma}
\begin{proof}
Let $\mc{S} = \{\xi_n\}_{n = 1}^\infty$. For $(x, y), (x, z) \in \R$ and $n, m \ge 1$, we have
\begin{align*}
\<\pi(x, y)\xi_n(y) \sten 1_{\{y\}}, \pi(x, z)\xi_m(z) \sten 1_{\{z\}}\>
&= \<\pi(x, y)\xi_n(y), \pi(x, z)\xi_m(z)\> \cdot 1_{\{y\}}(z) \\
&= \<\pi(x, y)\xi_n(y), \pi(x, y)\xi_m(y)\> \cdot 1_{\{y\}}(z) \\
&= \<\xi_n(y), \xi_m(z)\> \cdot 1_{\{y\}}(z) \\
&= \<\xi_n(y) \sten 1_{\{y\}}, \xi_m(z) \sten 1_{\{z\}}\>.
\end{align*}
Since $\H_x \ten \ell^2([x]_\R) = \ol{\spn \{ \xi_n(x) \sten 1_{\{y\}} : (x, y) \in \R, n \ge 1\}}$ for each $x \in X$, the above calculation shows that the formula
\begin{align*}
U_x(\xi_n(x) \sten 1_{\{y\}}) = \pi(x, y)\xi_n(y) \sten 1_{\{y\}} \text{ for } (x, y) \in \R, n \ge 1
\end{align*}
gives rise to a well defined unitary $U_x \in \U( \H_x \ten \ell^2([x]_\R))$ (note that $U_x$ is surjective since $\{\pi(x, y)\xi_n(y)\}_{n = 1}^\infty$ is a basis for $\H_x$ for $(x, y) \in \R$) for each $x \in X$. Moreover, for $(x, y), (x, z) \in \R$ and $n \ge 1$ we have
\begin{align*}
U_z([\id_{\mc{S}} \sten \lambda](z, x) \cdot \xi_n(x) \sten 1_{\{y\}}) 
= U_z(\xi_n(z) \sten 1_{\{y\}}) 
= \pi(z, y)\xi_n(y) \sten 1_{\{y\}} \\
= [\pi \sten \lambda](z, x) \cdot \pi(x, y)\xi_n(y) \sten 1_{\{y\}} 
= [\pi \sten \lambda](z, x) \cdot U_x(\xi_n(x) \sten 1_{\{y\}})
\end{align*}
and hence $[\pi \sten \lambda](z, x) U_x = U_z [\id_{\mc{S}} \sten \lambda](z, x)$ for $(z, x) \in \R$. 

For measurability, take any $g, h \in [\R]$, $n, m \ge 1$ and note that 
\begin{align*}
x &\mapsto \<U_x(\xi_n(x) \sten 1_{\{g^{-1}x\}}), \xi_m(x) \sten 1_{\{h^{-1}x\}}\>  \\
&= \<\pi(x, g^{-1}x)\xi_n(g^{-1}x) \sten 1_{\{g^{-1}x\}}, \xi_m(x) \sten 1_{\{h^{-1}x\}}\> \\
&= \<\pi(x, g^{-1}x)\xi_n(g^{-1}x), \xi_m(x)\> \cdot 1_{\{y \in X: g^{-1}y = h^{-1}y\}}(x)
\end{align*}
is the product of two measurable maps. 
\end{proof}

\begin{lemma}\label{L: bimodule} The Hilbert space $\H^\pi$ has an $L(\R)$-$L(\R)$ bimodule structure which satisfies
\begin{align}\label{E: bimodule}
au_{g} \cdot (\xi \sten_A \eta) \cdot x = \tilde\pi_g(\xi) \sten_A au_{g}\eta x 
\end{align}
for $a \in A$, $g \in [\R]$, $x \in M$, $\eta \in L^2(M)$, and $\xi \in \int_X^{\oplus} \H_x d\mu$.
\end{lemma}
\begin{proof}
We have already from the construction of Connes' fusion tensor that $\H^\pi$ is an $A$-$L(\R)$ bimodule with the right action satisfying \eqref{E: bimodule}. 
The proposed left and right actions certainly commute, so it is enough to show that the left action in \eqref{E: bimodule} makes $\H^\pi$ into a left Hilbert $L(\R)$-module. 
For each $n \ge 1$, set $p_n = 1_{\{x \in X: \dim \H_x \ge n\}} \in A$. If $(x, y) \in \R$, then $\H_x = \pi(x, y)\H_y$ so $\dim \H_x = \dim \H_y$ and therefore $p_n \in \Z(L(\R))$. Let $\{\xi_n\}_{n = 1}^\infty$ be an orthonormal fundamental sequence of sections for $X \ast \H$ and note that $p_n = \|\xi_n(\cdot)\|$. Set $\K = \bigoplus_{n = 1}^\infty p_nL^2(\R)$. 
We wish to define a unitary $U: \bigoplus_{n = 1}^\infty p_nL^2(\R) \to \H^\pi$. For any $g \in [\R]$, $a \in A$, and $n \ge 1$, let $\eta_{n, a, g} \in \K$ denote the vector which is $p_nau_g$ in the $n$th summand and $0$ elsewhere (note that $p_nau_g \in p_mL^2(\R)$ for any $1 \le m \le n$, so we must be careful with our notation). Then $\K = \ol{ \spn \{\eta_{n, a, g} : a \in A, g \in [\R], n \ge 1\} }$ and we define $U(\eta_{n, a, g}) = \tilde \pi_g (\xi_n) \sten_A au_g$. Then for $a, b \in A$, $g, h \in [\R]$, and $n \ge 1$, since $\E_A(u_gu_h^*) = 1_{\{x \in X: g^{-1}x = h^{-1}x\}}$, we have
\begin{align*}
\<\tilde \pi_g(\xi_n) \sten_A au_g, \tilde \pi_h(\xi_n) \sten_A bu_h\> 
 &= \tau(\<\tilde \pi_g(\xi_n)(\cdot), \tilde \pi_h(\xi_n)(\cdot)\> a\E_A(u_gu_h^*)b^*) \\
 &= \tau([u_g\|\xi_n(\cdot)\|^2u_g^*] a\E_A(u_gu_h^*)b^*) \\
 &= \tau(p_n au_gu_h^*b^*) \\
 &= \<\eta_{n, a, g}, \eta_{n, b, h}\>
\end{align*}
and if $m \ne n$, then $\<\tpi_g(\xi_n)(\cdot), \tpi_h(\xi_m)(\cdot)\>\E_A(u_gu_h^*) = 0$ and therefore
\begin{align*}
\<\tilde \pi_g(\xi_n) \sten_A au_g, \tilde \pi_h(\xi_m) \sten_A bu_h\> = 0 = \<\eta_{n, a, g}, \eta_{m, b, h}\>.
\end{align*} 
Thus $U$ extends to a well defined unitary $U: \K \to \H^\pi$ ($U$ is surjective since $\int_X^{\oplus} \H_x d\mu = \ol{\spn \{\tpi_g(\xi_n)a : a \in A, n \ge 1\}}$ for each $g \in [\R]$). 

Now since $\K = \bigoplus_{n = 1}^\infty p_nL^2(\R)$ is a left $L(\R)$-module by left multiplication in each coordinate, $\H^\pi$ becomes a left $L(\R)$-module under the action $x \cdot \eta = U(x \cdot U^*(\eta))$. Moreover, for $a, b \in A$, $g, h \in [\R]$, $n \ge 1$, we have
\begin{align*}
au_g \cdot (\tpi_h(\xi_n) \sten_A bu_h) 
&= au_g \cdot U(\eta_{n, b, h}) 
= U(au_g \cdot \eta_{n, b, h}) 
= U(\eta_{n, a(u_gbu_g^*), gh}) \\ 
&= \tpi_{gh}(\xi_n) \sten_A a(u_gbu_g^*)u_{gh} 
= \tpi_g(\tpi_h(\xi_n)) \sten_A au_gbu_h.
\end{align*} 
For any $a, b \in A$, $g, h \in [\R]$, $\xi \in \int_X^{\oplus} \H_x d\mu$, write $\xi = \sum_{n = 1}^\infty \tpi_h(\xi_n)a_n$ with $a_n \in A$. Then using the above,
\begin{align*}
au_g \cdot (\xi \sten_A bu_h) 
&=  \sum_{n = 1}^\infty au_g \cdot (\tpi_h(\xi_n)a_n \sten_A bu_h)
= \sum_{n = 1}^\infty \tpi_g(\tpi_h(\xi_n)) \sten_A au_ga_nbu_h) \\
&= \sum_{n = 1}^\infty \tpi_g(\tpi_h(\xi_n))(u_ga_nu_g^*) \sten_A au_gbu_h)
= \sum_{n = 1}^\infty \tpi_g(\tpi_h(\xi_n)a_n) \sten_A au_gbu_h) \\
&= \tpi_g(\xi) \sten_A au_gbu_h.
\end{align*}
Since elements of the form $bu_h$ span a dense subspace of $L^2(\R)$, it follows that the left action of $L(\R)$ satisfies \eqref{E: bimodule}.  
\end{proof}

Given two representations $\pi$ and $\rho$ of $\R$ with $\pi$ weakly contained in (resp. unitarily equivalent to) $\rho$, one can check that $\H^\pi$ is weakly contained in (resp. unitarily equivalent to) $\H^\rho$ as $M$-$M$ bimodules. If a representation $\pi$ is a mixing, then $\H^\pi$ is mixing relative to $A$. 

\subsection{Proof of Theorem \ref{T: main}}\label{SS: main}
We can now prove the main primeness result.  

\begin{repmytheorem}{T: main}
Let $\R$ be a countable pmp equivalence relation with no amenable direct summand which admits an unbounded 1-cocycle into a mixing orthogonal representation weakly contained in the regular representation. Then $L(\R) \ncong N \ten Q$ for any type ${\rm II}$ von Neumann algebras $N$ and $Q$ and
hence $\R \ncong \R_1 \times \R_2$ for any pmp $\R_i$ which have a.e. equivalence class infinite.
In particular, if $\R$ is ergodic then $L(\R)$ is prime. 
\end{repmytheorem}
\begin{proof}
Consider the $s$-malleable deformation $M \subset \tilde M$, $\{\alpha_t\}_{t \in \RR} \subset \Aut(\tilde M)$ constructed in Section \ref{S: deform}. 
Note that the representation $\hat \pi = \oplus_{n = 1}^\infty \pi_\CC^{\odot n}$ is mixing and weakly contained in the regular representation $\lambda$ of $\R$, since $\pi$ has these properties.  

By identifying $L^2(X \ast \Omega) \ominus L^2(X)$ with $\int_X^{\oplus} [L^2(\Omega_x) \ominus \CC] d\mu(x)$, one can check that
$L^2(\tilde M) \ominus L^2(M) \cong \H^{\rho}$ as $M$-$M$ bimodules. 
The latter is then unitarily equivalent to $\H^{\hat \pi}$ by Lemma \ref{L: chaos}, and since $\hat \pi$ is mixing, we have that $L^2(\tilde M) \ominus L^2(M)$ is mixing relative to $A$. Moreover, $\H^{\hat \pi} \prec \H^{\lambda}$ since $\hat \pi \prec \lambda$, and one can check that $\H^\lambda \cong L^2(M) \sten[A] L^2(M)$ as $M$-$M$ bimodules. Since $A$ is amenable, the latter is weakly contained in the coarse $M$-$M$ bimodule.

Since $b$ is unbounded, there is $\delta > 0$ such that for all $R > 0$, there is $g \in [\R]$ with $\mu(\{\|b(x, g^{-1}x)\| \ge R\}) \ge \delta$. If we had $\alpha_t \to \id$ uniformly on $(M)_1$, there would be $t_0$ such that $\|\E_M(\alpha_{t_0}(u_g))\|^2_2 > 1-\frac{\delta}{2}$ for all $g \in [\R]$. 
But taking $R > 0$ large enough that $e^{-2t_0^2R^2} < \frac{\delta}{2}$, and $g \in [\R]$ with $\mu(\{\|b(x, g^{-1}x)\| \ge R\}) \ge \delta$, we would have
\begin{align} \label{E: contra}
\begin{split}
1 - \frac{\delta}{2} 
&< \|\E_M(\alpha_{t_0}(u_g))\|^2_2 
= \int_X e^{-2t_0^2\|b(x, g^{-1}x)\|^2} d\mu  \\
 &\le \mu(\{\|b(x, g^{-1}x)\|^2 < R\}) + e^{-2t_0^2R^2} \mu(\{\|b(x, g^{-1}x)\|^2 \ge R\}) \\
 &< 1-\delta + \frac{\delta}{2}\cdot 1
 = 1 - \frac{\delta}{2}
 \end{split}
\end{align}
which is false. Hence $\alpha_t \to \id$ is not uniform on $(M)_1$, and so by Theorem \ref{T: deduce} we conclude that $M \ncong N \ten Q$ for $N$ and $Q$ of type ${\rm II}$.  

In particular, if $\R \cong \R_1 \times \R_2$, then $L(\R) \cong L(\R_1) \ten L(\R_2)$, so 
there is $j \in \{1, 2\}$ such that $L(\R_j)$ is not type ${\rm II}$ and hence $\R_j$ does not have a.e. equivalence class infinite.
\end{proof}

\subsection{Remark} Theorem \ref{T: main} (as well as Theorem \ref{T: factorR}) in fact holds with $L(\R)$ replaced by $L(\R, \sigma)$, which is constructed as $L(\R)$, but ``twisted" by some 2-cocycle $\sigma: [\R] \times [\R] \to \U(L^\infty(X))$, in the sense that for $g, h \in [\R]$ the unitaries $u_g, u_h, u_{gh} \in L(\R, \sigma)$ satisfy $u_gu_h = \sigma(g, h)u_{gh}$. Indeed, with $\tilde \R$ and \eqref{E: multiplicativecocycle} exactly as before, the formula \eqref{E: auto} now gives rise to an $s$-malleable deformation of $L(\R, \sigma) \subset L(\tilde \R, \sigma)$. Similarly, \eqref{E: bimodule} now defines an $L(\R, \sigma)$-$L(\R, \sigma)$ bimodule and the necessary identifications in the proof of Theorem \ref{T: main} hold. 

There is good reason for considering the algebras $L(\R, \sigma)$. The subalgebra $L^\infty(X) \subset L(\R, \sigma)$ is a \emph{Cartan subalgebra}, i.e., it is maximal abelian and its normalizer generates $L(\R, \sigma)$ as a von Neumann algebra. Such subalgebras have been the object of intense study (see \cite{Io12} for a detailed survey). Feldman and Moore showed in \cite{FM75b} that a Cartan subalgebra $A \subset M$ of a tracial von Neumann algebra $M$ always arises as $L^\infty(X) \subset L(\R, \sigma)$ for some 2-cocycle $\sigma$ and measured equivalence relation $\R$ on a standard probability space $X$.

\vspace{.5cm}
\section{Unique Prime Factorization}\label{S: factor}

In this section we obtain a unique prime factorization result for a class of type $\rm{II}_1$ factors in the spirit of \cite{OP03}. 
It is important to note that for $\rm{II}_1$ factors $N, Q$, and any $t > 0$ we have $N \ten Q \cong N^t \ten Q^{1/t}$, so unique factorizations are considered modulo amplifications as well as unitary conjugacy. 

\subsection{An Obstruction to Unique Factorization}\label{SS: necfactor}
We will need two lemmas before our proof of Theorem \ref{T: necfactor}. Both are well-known, but we include their proofs for completeness. 

\begin{lemma} \label{L: closep} 
Let $N \subset M$ be an inclusion of tracial von Neumann algebras. For any $\epsilon > 0$ and projection $p \in M$ satisfying $\|p - \E_N(p)\|_2 < \epsilon$, there exists a projection $q \in N$ such that $\|p - q\|_2 < \epsilon + \sqrt{10}\epsilon^{1/3}$.
\end{lemma}
\begin{proof}
Note that 
\begin{align*}
\|\E_N(p)^2 - \E_N(p)\|_2 
\le \|\E_N(p)^2 - p\E_N(p)\|_2 + \|p(\E_N(p) - p)\|_2 + \|p - \E_N(p)\|_2
< 3\epsilon
\end{align*}
and therefore for any $\delta > 0$, Chebyshev's inequality gives
\begin{align*}
\tau(1_{\{(\delta, 1-\delta)\}}(\E_N(p))) 
\le \tau(1_{\{|t^2 - t| > \delta^2\}}(\E_N(p))) 
\le \frac{1}{\delta^4} \|\E_N(p)^2 - \E_N(p)\|_2^2 
\le \frac{9\epsilon^2}{\delta^4}
\end{align*}
so that $q = 1_{\{|t-1| \le \delta\}}(\E_N(p))$ satisfies 
$
\|\E_N(p) - q\|_2^2 \le \frac{9\epsilon^2}{\delta^4} + \delta^2
$.
Taking $\delta = \epsilon^{1/3}$ then gives
\begin{align*}
\|p - q\|_2 \le \|p - \E_N(p)\|_2 + \|\E_N(p) - q\|_2 \le \epsilon + \sqrt{10}\epsilon^{1/3}
\end{align*}
\end{proof}

\begin{lemma} \label{L: pcomm}
Let $\{p_n\}_{n = 1}^\infty \subset M$ be an asymptotically central sequence of projections. Then there exist commuting projections $\{q_k\}_{k = 1}^\infty \subset M$ which are asymptotically central with $\|p_{n_k} - q_k\|_2 \to 0$ as $k \to \infty$ for some subsequence $\{p_{n_k}\}_{k = 1}^\infty$. 
\end{lemma}
\begin{proof}
Let $q_1 = p_1$. Then given $q_1, \dots, q_k$, commuting projections, we let $A = \{q_1, \dots, q_k\}''$, and note that $A = \bigoplus_{i = 1}^{m} \CC e_i $ for projections $\{e_i\}_{i = 1}^m$ which are minimal in $A$ and such that $\sum_{i = 1}^m e_i = 1$. Let $n_{k+1}$ be large enough so that $\|p_{n_{k+1}}e_i - e_ip_{n_{k+1}}\|_2^2 < \frac{1}{mk}$ for each $1 \le i \le m$. 
Then $A' \cap M = \bigoplus_{i = 1}^{m} e_iMe_i$ and $E_{A' \cap M}(x) = \sum_{i = 1}^m e_ixe_i$ for $x \in M$ and hence
\begin{align}
\|p_{n_{k+1}} - \E_{A' \cap M}(p_{n_{k+1}})\|_2^2 
&= \|p_{n_{k+1}} - \sum_{i = 1}^m e_ip_{n_{k+1}} e_i\|_2^2 
= \sum_{i = 1}^m \|e_ip_{n_{k+1}} - e_ip_{n_{k+1}}e_i\|_2^2 \\
&\le \sum_{i = 1}^m \|e_ip_{n_{k+1}} - p_{n_{k+1}}e_i\|_2^2 
< \sum_{i = 1}^m \frac{1}{mk} 
= \frac{1}{k}
\end{align}
Then Lemma \ref{L: closep} gives $q_{k+1} \in A' \cap M$ with $\|p_{n_{k+1}} - q_{k+1}\|_2 \le \frac{1}{k} + \frac{\sqrt{10}}{k^{1/3}}$. Thus $\|p_{n_k} - q_k\|_2 \to 0$ as $k \to \infty$ from which it further follows that the sequence $\{q_k\}_{k = 1}^\infty$ is asymptotically central.
\end{proof}

\begin{repmytheorem}{T: necfactor}
Let $M_1$ and $M_2$ be $\|\cdot\|_2$-separable $\rm{II}_1$ factors with property Gamma and set $M = M_1 \ten M_2$. Then there is an approximately inner automorphism $\phi \in \ol{\text{Inn}(M)}$ such that $\phi(M_i) \nprec M_j$ for any $i, j \in \{1, 2\}$.
\end{repmytheorem}

\begin{proof}
Since $M_1$ and $M_2$ have $\Gamma$, there exist asymptotically central sequences of projections $\{p_k\}_{k = 1}^\infty \subset M_1$, $\{q_k\}_{k = 1}^\infty \subset M_2$ such that $\tau(p_k), \tau(q_k) \to  \frac{1}{2}$ as $k \to \infty$. 
Using Lemma \ref{L: pcomm} we may assume that $[p_k, p_j] = [q_k, q_j] = 0$ for all $k, j \ge 1$. 
Let $\{x_i\}_{i = 1}^\infty \subset M$ be a $\|\cdot\|_2$-dense sequence. 

{\it Claim:} There are sequences $\{v_n\}_{n = 1}^\infty \subset \U(M_1)$, $\{u_n\}_{n = 1}^\infty \subset \U(M_2)$ and a subsequence $\{k_n\}_{n = 1}^\infty \subset \NN$ such that for each $n \ge 1$, the asymptotically central unitaries $w_n = 1 - 2p_{k_n}q_{k_n}$ satisfy
\begin{align}
\|w_nv_nw_n^* - (1-2q_{k_n})v_n\|_2 &\le \frac{1}{2^{n}} \quad\text{ and }\quad
\|w_nu_nw_n^* - (1-2p_{k_n})u_n\|_2 \le \frac{1}{2^{n}}, \label{E: move} \\
\|w_iv_nw_i^* - v_n\|_2 &\le \frac{1}{2^{n}} \quad\text{ and }\quad
\|w_iu_nw_i^* - u_n\|_2 \le \frac{1}{2^{n}} \quad \text{for} \quad 1 \le i < n, \label{E: movesmart} \\
\|w_nv_iw_n^* - v_i\|_2 &\le \frac{1}{2^{n}}  \quad\text{ and }\quad
\|w_nu_iw_n^* - u_i\|_2 \le \frac{1}{2^{n}} \quad \text{for} \quad 1 \le i < n, \label{E: conv2} \\
\|w_nx_iw_n^* - x_i\|_2 &\le \frac{1}{2^{n}} \quad \text{ for} \quad 1 \le i < n. \label{E: conv1}
\end{align}

Before we prove the claim, let us prove the theorem assuming it holds. For $n \ge 1$, let $W_n = w_1w_2\cdots w_n$. Then $W_nx_iW_n^*$ is $\|\cdot\|_2$-Cauchy for any $i \ge 1$, since for any $m \ge n > i$, 
\begin{align*}
\|W_mx_iW_m^* - W_nx_iW_n^*\|_2 = \|w_{n+1}\cdots w_{m} x_i w_m^* \cdots w_{n+1}^* - x_i\|_2
\le \sum_{j = n+1}^m \frac{1}{2^j} < \frac{1}{2^n}
\end{align*}
using \eqref{E: conv1}. Similarly, \eqref{E: conv1} implies that $W_n^*x_iW_n$ is also\footnote{In fact $W_n = W_n^*$ here, but it is useful to note this as an implication of \eqref{E: conv1} since the construction $W_n = w_1w_2\cdots w_n$ can be done without arranging $w_n = w_n^*$ and $[w_n, w_m] = 0$ for all $n, m \ge 1$.} $\|\cdot\|_2$-Cauchy for any $i \ge 1$, so we may define $\phi \in \ol{\text{Inn}(M)}$ by $\phi(x) = \lim\limits_{n \to \infty} W_nxW_n^*$ for $x \in M$, where the convergence is in the SOT. Then noting that $[w_n, w_m] = 0$ for all $n, m \ge 1$, and using \eqref{E: movesmart} and \eqref{E: conv2}, for each $n \ge 1$ we have
\begin{align*}
\|\phi(v_n) - w_nv_nw_n^*\|_2 
&= \lim_{k \to \infty} \|w_1w_2\cdots w_kv_nw_k^* \cdots w_2^*w_1^* - w_nv_nw_n^*\|_2 \\
&\le \limsup_{k \to \infty} \left[\sum_{i = 1}^{n-1}\|w_iv_nw_i^* - v_n\|_2 + \sum_{i = n+1}^{k}\|w_iv_nw_i^* - v_n\|_2\right] \\
&\le \limsup_{k \to \infty} \left[\sum_{i = 1}^{n-1}\frac{1}{2^n} + \sum_{i = n+1}^{k}\frac{1}{2^i}\right]
\le \frac{n-1}{2^n} + \frac{1}{2^n} = \frac{n}{2^n}
\end{align*}
which, combined with \eqref{E: move}, gives
\begin{align} \label{E: want}
\|\phi(v_n) - (1 - 2q_{k_n})v_n\|_2 \le \frac{n+1}{2^n} \quad \text{for all} \quad n \ge 1.
\end{align}
We then see that for any $y \in M_2$,
\begin{align*}
&\limsup_{n \to \infty} \|\E_{M_1}(y\phi(v_n))\|_2 
= \limsup_{n \to \infty} \|E_{M_1}(y(1 - 2q_{k_n})v_n)\|_2 \\
&=  \limsup_{n \to \infty} \|\tau(y(1 - 2q_{k_n}))v_n\|_2
= \limsup_{n \to \infty} |\tau(y)(1 - 2\tau(q_{k_n}))| \\
&= |\tau(y)(1 - 2\cdot \frac{1}{2})| 
= 0
\end{align*}
where we use the fact that $\tau(yy_n) - \tau(y)\tau(y_n) \to 0$ for any $y \in M_2$ and any asymptotically central sequence $\{y_n\} \subset M_2$, which follows from the uniqueness of the trace. Since $M = M_1 \ten M_2$, this calculation then implies that $\|\E_{M_1}(a\phi(v_n)b)\|_2 \to 0$ for all $a, b \in M$ and hence $\phi(M_1) \nprec M_1$. The same argument for the sequence $\{\phi(u_n)\}$ shows that $\phi(M_2) \nprec M_2$.

On the other hand, note that $\phi(p_k) = p_k$ for all $k \ge 1$, so that for each $x \in M_1$,
\begin{align*}
&\limsup_{k \to \infty} \|\E_{M_2}(\phi(1-2p_k)x)\|_2 
= \limsup_{k \to \infty} \|\E_{M_2}((1-2p_k)x)\|_2 \\
&= \limsup_{k \to \infty} |\tau((1-2p_k)x)|
= \limsup_{k \to \infty} |(1-2\tau(p_k))\tau(x)|\\
&= |(1 - 2\cdot \frac{1}{2})\tau(x)| 
= 0
\end{align*}
so that $\|\E_{M_2}(a\phi(1-2p_k)b)\|_2 \to 0$ for all $a, b \in M$ and hence $\phi(M_1) \nprec M_2$. Similarly, analyzing the sequence $\{\phi(1-2q_k)\}$ shows that $\phi(M_2) \nprec M_1$.

{\it Proof of Claim:} We construct the necessary sequences recursively. Therefore, suppose we are given $\{k_1, \dots, k_{n-1}\}$, $\{v_1, \dots, v_{n-1}\}$, and $\{u_1, \dots, u_{n-1}\}$ such that \eqref{E: move}, \eqref{E: movesmart}, \eqref{E: conv1}, and \eqref{E: conv2} are satisfied (allowing these sets to be empty for the base case $n = 1$). We construct $k_n$, $v_n$, and $u_n$ as follows. 

Letting $B_1 = \{p_{k_1}, \dots, p_{k_{n-1}}\}''$, we know that $B_1$ is abelian and hence of the form $B_1 = \bigoplus_{i = 1}^{m}  \CC e_i $ for projections $\{e_i\}_{i = 1}^m$ which are minimal in $B_1$ and such that $\sum_{i = 1}^m e_i = 1$. Then $B_1' \cap M_1^\omega = \bigoplus_{i = 1}^{m} e_iM_1^\omega e_i$ and therefore $\Z(B_1' \cap M_1^\omega) = B_1$ since $M_1$ is a factor. 
Let $p$ denote the image of the sequence $\{p_k\}$ in $M_1' \cap M_1^\omega$, noting that $\tau(p) = \frac{1}{2} = \tau(1-p)$. Then $p \sim (1-p)$ in $B_1' \cap M_1^\omega$ since
  \begin{align*}
 \E_{\Z(B_1' \cap M_1^\omega)}(p) 
 &= \E_{B_1}(p) 
 = \sum_{i = 1}^m \frac{\tau(pe_i)}{\tau(e_i)}e_i  
 = \sum_{i = 1}^m \frac{\tau(p)\tau(e_i)}{\tau(e_i)}e_i  \\
 &= \tau(p)
 = \tau(1-p) = \E_{\Z(B_1' \cap M_1^\omega)}(1-p). 
 \end{align*}
 
Thus there is $\tilde v \in \U(B_1' \cap M_1^\omega)$ such that $\tilde vp\tilde v^* = 1-p$. Setting $B_2 = \{q_{k_1}, \dots, q_{k_{n-1}}\}''$ and letting $q$ denote the image of $\{q_k\}$ in $M_2' \cap M_2^{\omega}$, the same argument shows that there is $\tilde u \in \U(B_2' \cap M_2^{\omega})$ such that $\tilde uq\tilde u^* = 1-q$.
Lifting $\tilde v$ and $\tilde u$ to sequences of unitaries $\{\tilde v_k\}_{k = 1}^\infty \subset \U(M_1)$, $\{\tilde u_k\}_{k = 1}^\infty \subset \U(M_2)$ which asymptotically commute with $B_1$ and $B_2$, we can then find $k_n$ large enough that $v_n = \tilde v_{k_n}$ and $u_n = \tilde u_{k_n}$ have 
\begin{align} \label{E: inter}
\|v_np_{k_n}v_n^* - (1-p_{k_n})\|_2 &\le \frac{1}{2^{n+1}}, \quad\;\; \quad \quad
\|u_nq_{k_n}u_n^* - (1-q_{k_n})\|_2 \le \frac{1}{2^{n+1}}, \\
\|v_np_{k_i}v_n^* - p_{k_i}\|_2 &\le \frac{1}{2^{n+1}}, \quad \text{and} \quad
\|u_nq_{k_i}u_n^* - q_{k_i}\|_2 \le \frac{1}{2^{n+1}} \quad \text{for} \quad 1 \le i < n,
\end{align}
and we further assume that $k_n$ is large enough that \eqref{E: conv1} and \eqref{E: conv2} are satisfied (which can be done since $\{(1 - 2p_kq_k)\}_{k = 1}^\infty$ is asymptotically central). 
Noting that
\begin{align*}
[1 - 2p_{k_n}q_{k_n}][1- 2(1-p_{k_n})q_{k_n}] = 1 - 2q_{k_n},
\end{align*}
from \eqref{E: inter} we get
\begin{align*}
\|w_nv_nw_n^* - [1-2q_{k_n}]v_n\|_2 = \|[1 - 2p_{k_n}q_{k_n}][1- 2(v_np_{k_n}v_n^*)q_{k_n}] - [1 - 2q_{k_n}]\|_2 \\
\le  \|2(1-p_{k_n})q_{k_n} - 2(v_np_{k_n}v_n^*)q_{k_n}\|_2 
\le 2\cdot \frac{1}{2^{n+1}} = \frac{1}{2^n},
\end{align*}
and similarly $\|w_nu_nw_n^* - [1-2p_{k_n}]u_n\|_2 \le \frac{1}{2^n}$ so that \eqref{E: move} is satisfied. For $1 \le i < n$, we use \eqref{E: inter} to estimate
\begin{align*}
\|v_nw_i^*v_n^* - w_i^*\|_2 = \|[1-2(v_np_{k_i}v_n^*)q_{k_i}] - [1 - 2p_{k_i}q_{k_i}]\|_2 \\
\le 2\|v_np_{k_i}v_n^* - p_{k_i}\|_2 \le  2\cdot \frac{1}{2^{n+1}} = \frac{1}{2^n},
\end{align*}
and similarly $\|u_nw_i^*u_n^* - w_i^*\|_2 \le \frac{1}{2^n}$ so that \eqref{E: movesmart} holds.  

\end{proof}

\subsection{Unique Prime Factorization via $s$-Malleable Deformation}

The principle challenge in the proof of the unique prime factorization in Theorem \ref{T: factor} is controlling the Cartan subalgebras of each factor. The following proposition will be critical for this reason:

\begin{prop}\label{P: tentwine}
Let $M = N \ten Q = M_1 \ten M_2$ be a ${\rm II}_1$ factor without property Gamma, and suppose that $N \prec_M A \ten M_2$ for some Cartan subalgebra $A$ of $M_1$. Then there is $t > 0$ and $u \in \U(M)$ such that $uN^tu^* \subset M_2$
under the identification $N \ten Q = N^t \ten Q^{1/t}$.
\end{prop}
\begin{proof}
By Theorem \ref{T: twine} there are projections $p \in N$, $f \in A \ten M_2$, a unital normal $*$-homomorphism $\theta: pNp \to f(A \ten M_2)f$, and a nonzero partial isometry $v \in M$, such that 
\begin{align}   \label{E: twine1}
\theta(x)v = vx \quad \text{for all} \quad x \in pNp, \quad 
v^*v \in (N' \cap M)p, \quad \text{and} \quad
vv^* \in \theta(pNp)' \cap fMf   
\end{align}
Let $L = \theta(pNp)' \cap fMf$,
$\Z = \Z(L)$,  and $e = vv^*$. Note that $Af \subset L$ and therefore $\Z \subset (Af)' \cap fMf = f(A \ten M_2)f$. From \eqref{E: twine1} it follows that 
\begin{align*}
v^*\Z v \subset \Z(v^*v(N' \cap M)v^*v) = \Z(Q)v^*v = \CC v^*v 
\end{align*}
and hence $\Z e = v(\CC v^*v)v^* = \CC e$. 
Therefore setting $z = C(e)$ (the support of $e$ in $\Z$), and taking any $z' \in \Z$, $z ' \le z$, we have $z' e \in \CC e$ and hence $z'e \in \{0, e\}$ which implies that $z' \in \{0, z\}$. Thus $Lz$ is a finite factor. Hence there is $e' \in Lz$, $e' \le e$ with $\tau_{Lz}(e') = \tau(e')/\tau(z) = \frac{1}{n}$ for some integer $n$. Let $v_1 = e'v$ and note that for any $x \in pNp$ we have 
\begin{align}
v_1^*v_1x = v^*e'vx = v^*e'\theta(x)v = v^*\theta(x)e'v = xv^*e'v = xv_1^*v_1
\end{align}
and hence $v_1^*v_1 \in (pNp)' \cap pMp = Qp$, so let $q \in Q$ be a projection such that $v_1^*v_1 = q \sten p$. 

Let $s = \tau(q)\tau(z)/\tau(e')$ and identify $Q \ten N = Q^s \ten N^{1/s}$ such that
$\CC q \ten pNp = \CC q' \ten p'N^{1/s}p'$ and
 $qQq \ten \CC p = q'Q^sq' \ten \CC p'$ for projections $q' \in Q^s$, $p' \in N^{1/s}$ with 
$\tau(q') = \tau(q)/s = \tau(e')/\tau(z) = \frac{1}{n}$ and 
$\tau(p') = \tau(p)s = \tau(z)$.
  
Since $Q^s$ and $Lz$ are factors, let $w_1, \dots, w_n \in Q^t$ and $u_1, \dots, u_n \in Lz$ be partial isometries with $\sum_{j = 1}^n w_jw_j^* = 1$, $\sum_{j = 1}^n u_ju_j^* = z$ and $w_j^*w_j = q'$, $u_j^*u_j = e'$ for all $1 \le j \le n$. 
 Then setting $w = \sum_{j = 1}^n u_jv_1w_j^*$ we have $w^*w = p'$ and $ww^* = z$, and $wN^{1/s}w^* \subset z(A \ten M_2)z$. 
 
Cutting $w$ to the right by a projection in $N$ under $p'$, we may assume that $\tau(p') = \tau(z) = \frac{1}{m}$ for some integer $m$. 
 By \cite[Theorem 3.2 and Remark 3.5.2]{Po81}, we can find a copy of the hyperfinite ${\rm II}_1$ factor $R$, with $A \subset R \subset M_1$ and $R' \cap M_1 = \CC$. 
Note that $Af \subset L \implies \Z \subset (Af)' \cap fMf = f(A \ten M_2)f \implies z \in A \ten M_2 \subset R \ten M_2$.
Since $R \ten M_2$ is a factor, there are partial isometries $\tilde u_1, \dots, \tilde u_m \in R \ten M_2$ with $\tilde u_j^*\tilde u_j = z$ for each $j$ and $\sum_{j = 1}^m \tilde u_j\tilde u_j^* = 1$. Taking partial isometries $\tilde w_1, \dots, \tilde w_m \in N$ with $\tilde w_j^*\tilde w_j = z$ for each $j$ and $\sum_{j = 1}^m \tilde w_j\tilde w_j^* = 1$, and setting $u_0 = \sum_{j = 1}^m \tilde u_jw\tilde w_j^*$ we have $u_0 \in \U(M)$ and $u_0N^{1/s}u_0^* \subset R \ten M_2$. 

Now write $R = \bigotimes_{j = 1}^\infty M_2(\CC)$ and set $R_k = \bigotimes_{j = k+1}^\infty M_2(\CC)$, so that $R = \left[\bigotimes_{j = 1}^{k} M_2(\CC) \right] \ten R_k$ for any $k \ge 1$. 
Then for any $\epsilon > 0$, there is $k \ge 1$ such that $\|b - \E_{u_0Q^su_0^*}(b)\|_2 < \epsilon$ for all $b \in \U(R_k)$. Indeed if not, there would be $\epsilon > 0$ and $\{b_k\} \subset \U(R)$ with $b_k \in \U(R_k)$ and $\|b_k - \E_{u_0Q^su_0^*}(b_k)\|_2 \ge \epsilon$ for all $k$. Then $\{b_k\}$ would be an asymptotically central sequence in $R$ and hence in $R \ten M_2$. In particular, $\{b_k\}$ would asymptotically commute with $u_0N^{1/s}u_0$ which does have property Gamma since $M$ is non-Gamma. But this would imply that $\|b_k - \E_{u_0Q^su_0^*}(b_k)\|_2 \to 0$ by Connes characterization of property Gamma in \cite{Co75b}, a contradiction. 

In particular, taking $\epsilon = \frac{1}{2}$ we find $k \ge 1$ such that $\|b - \E_{u_0Q^su_0^*}(b)\|_2 < \frac{1}{2}$ for all $b \in \U(R_k)$ which implies that $R_k \prec u_0Q^su_0^*$. It follows that $R = M_{2^k}(R_k)$ has $R \prec u_0Q^su_0^*$. Using Lemma 3.5 of \cite{Va07}, we pass to relative commutants to find that $u_0N^{1/s}u_0^* \prec R' \cap M = M_2$ and then $M_1 \prec u_0Q^{s}u_0^*$.

Then by Proposition 12 of \cite{OP03}, since $M_1' \cap M = M_2$ is a factor, there is $r > 0$ and $\tilde u_0 \in \U(M)$ such that 
$\tilde u_0 M_1\tilde u_0^* \subset u_0Q^{sr}u_0^*$
after identifying $u_0(N^{1/s} \ten Q^{s})u_0^* \cong u_0(N^{1/sr} \ten Q^{sr})u_0^*$. Setting $t = 1/sr$ and $u = \tilde u_0^* u_0$, we have $uN^{t}u^* = (\tilde u_0^* u_0Q^{sr} u_0^* \tilde u_0)' \cap M \subset M_1' \cap M = M_2$. 
\end{proof}

\begin{theorem} \label{T: factor}
Let $M_1, \dots, M_k$ be $\rm{II}_1$ factors without property Gamma, each with an $s$-malleable deformation $\{\alpha^i_t\}_{t \in \RR} \subset \emph{Aut} (\tilde M_i)$ for some tracial von Neumann algebras $\tilde M_i \supset M_i$. Suppose that for each $i$, the $M_i$-$M_i$ bimodule $_{M_i}L^2(\tilde M_i) \ominus L^2(M_i)_{M_i}$ is weakly contained in the coarse $M_i$-$M_i$ bimodule and mixing relative to some abelian subalgebra $A_i \subset M_i$.  Assume that the convergence $\alpha^i_t \to \id$ is not uniform in $\|\cdot\|_2$ on $(M_i)_1$ for any $i$. Then $M_i$ is prime for each $i$, and

(1). If $M = M_1 \ten M_2 \ten \dots \ten M_k = N \ten Q$ for tracial factors $N, Q$, there must be a partition $I_N \cup I_Q = \{1, \dots, k\}$ and $t > 0$ such that  
$N^t = \bigotimes_{i \in I_N} M_i$ and $Q^{1/t} = \bigotimes_{i \in I_Q} M_i$ modulo unitary conjugacy in $M$.

(2). If $M = M_1 \ten M_2 \ten \dots \ten M_k = P_1 \ten P_2 \ten \cdots \ten P_m$ for $\rm{II}_1$ factors $P_1, \dots, P_m$ and $m \ge k$, then $m = k$, each $P_i$ is prime, and there are $t_1, \dots, t_k > 0$ with $t_1t_2\cdots t_k = 1$ such that after reordering indices and conjugating by a unitary in $M$ we have $M_i = P_i^{t_i}$ for all $i$. 

(3). In (2), the assumption $m \ge k$ can be omitted if each $P_i$ is assumed to be prime. 
\\
\end{theorem} 
\begin{proof}
We prove (1) by induction on $k$. Note that by Theorem \ref{T: deduce}, we know that each $M_i$ is prime, so the case $k = 1$ can only occur if either $Q$ or $N$ is finite dimensional. Without loss of generality, assume $N = M_n(\CC)$ for some $n \in \ZZ_{>0}$. Then $t = 1/n$ does the job with $I_N = \emptyset$.

Now suppose that $k \ge 2$ and for convenience set $M = N \ten Q$. 
Since $M$ is nonamenable, we assume without loss of generality that $Q$ is nonamenable. For each $i$, we extend $\alpha_t^i \in \Aut(\tilde M_i)$ to $\tilde M^i = M_1 \ten \cdots \ten \tilde M_i \ten \cdots \ten M_k$ by the rule $\alpha_t^i|_{M_j} = \id$ for $j \ne i$. Thus for each $i$ we obtain an $s$-malleable deformation $\{\alpha_t^i\}_{t \in \RR}$ of $M \subset \tilde M^i$. For $x \in M$, set $\delta_t^i(x) = \alpha^i_t(x) - \E_M(\alpha^i_t(x))$. For each $I \subset \{1, 2, \dots, k\}$, let $ M_I = \bigotimes_{j \in I} M_j$ and $\hat M_I = \bigotimes_{j \notin I} M_j$ so that $M = M_I \ten \hat M_I$.  

We claim that there must be $i$ such that $\alpha_t^i \to \id$ uniformly in $\|\cdot\|_2$ on $(N)_1$. Suppose not. Then using Lemma \ref{L: trans}, for each $i$ we find $\epsilon_i > 0$ and sequences $\{x_n^i\} \subset (N)_1$, $\{t^i_n\} \subset \RR$ with $t^i_n \to 0$ as $n \to \infty$ and  
$\xi^i_n = \delta^i_{t^i_n}(x_n^i) \in L^2(\tilde M^i) \ominus L^2(M)$ satisfying $\|\xi^i_n\| \ge \epsilon_i$, $\|x\xi^i_n\| \le \|x\|_2$, and $\|x\xi_n^i - \xi_n^i x\| \to 0$ as $n \to \infty$ for each $x \in Q$. 
Since $N = Q' \cap M$ is a factor, applying Lemma \ref{L: common} gives
\begin{align}\label{E: weak1}
{_ML^2(M)_{Q}} \prec {_ML^2(\tilde M^i) \ominus L^2(M)_{Q}}
\end{align}
But since ${_{M_i}[L^2(\tilde M_i) \ominus L^2(M_i)]_{M_i}} \prec {_{M_i}L^2(M_i) \ten L^2(M_i)_{M_i}}$ for each $i$, we also have 
\begin{align} \label{E: weak2}
{_M}[L^2(\tilde M^i) \ominus L^2(M)]{_M} \prec {_M}L^2(M) \sten[\hat M_i] L^2(M){_M}
\end{align}
for each $i$. Then combining \eqref{E: weak1} and \eqref{E: weak2} we have 
${_ML^2(M)_{Q}} \prec {_ML^2(M) \sten[\hat M_i] L^2(M)_{Q}}$,
so that $Q$ is amenable relative to $\hat M_i$ in $M$ for each $i$. But note that for any $I, J \subset \{1, 2, \dots, n\}$, the subalgebras $M_I$ and $M_J$ satisfy 
$M = \N_M(M_I)''$ and
$[e_{M_I}, e_{M_J}] = 0$, so that after $k-1$ applications of \ref{T: commuting square} we find that $Q$ is amenable relative to $\bigcap_{i = 1}^k \hat M_i = \CC$, which contradicts the nonamenability of $Q$.  
Thus there must indeed be some $j \in \{1, 2, \dots, k\}$ such that $\alpha_t^j \to \text{id}$ uniformly in $\|\cdot\|_2$ on $(N)_1$. 

We have that
$L^2(\tilde M^i) \ominus L^2(M)$ is mixing relative to $A_i \ten \hat M_i$ since $L^2(\tilde M_i) \ominus L^2(M_i)$ is mixing relative to $A_i$. 
It follows that there can be no sequence $\{u_n\} \subset (N)_1$ with $\|\E_{A_i \ten \hat M_i} (xu_ny)\|_2 \to 0$ for each $x, y \in M$. If there were, we would conclude, just as in \eqref{E: usemix}, that $\alpha^i_t \to \text{id}$ uniformly on $(Q)_1$, and then on all of $(M)_1$ as in the proof of Theorem \ref{T: deduce}. This would then contradict the assumption that the convergence $\alpha^i_t \to \text{id}$ is not uniform on $(M_i)_1$. 

Thus $N \prec_M A_i \ten \hat M_i$ by Theorem \ref{T: twine}. Then by Proposition \ref{P: tentwine}, there is $t > 0$ such that after decomposing $M = N^{t} \ten Q^{1/t}$ and conjugating by a unitary, we have $N^{t} \subset \hat M_i$. Set $P = (N^{t})' \cap \hat M_i = Q^{1/t} \cap \hat M_i$ so that $\hat M_i = N^{t} \ten P$. If $P$ is type $\rm{I}_n$ for some $n$, it follows that $\hat M_i = N^{nt}$ and $M_i = Q^{1/nt}$ and the proof is done. Otherwise, $P$ is type $\rm{II}_1$ and by the inductive hypothesis, there is a partition $I_{N} \cup I_P = \{1, \dots, k\} \setminus \{i\}$ and $s > 0$ such that $N^{st} = M_{I_N}$ and $P^{1/s} = M_{I_P}$ modulo unitary conjugation. Then since $Q^{1/st} = (N^{st})' \cap M = M_{I_N}' \cap M = M_i \ten M_{I_P}$,
setting $I_Q = I_P \cup \{i\}$ concludes the proof of (1).

We also prove (2) by induction on $k$. The case $k = 1$ follows immediately from the primeness of $M_1$. For $k \ge 2$, we apply (1) with $N = P_1 \ten \cdots \ten P_{m-1}$ and $Q = P_m$, to find a partition $I_N \cup I_Q = \{1, \dots, k\}$, $t > 0$ such that after conjugating by a unitary in $M$ we have $P_1^t \ten \cdots \ten P_{m-1}^t = M_{I_N}$ and $P_m^{1/t} = M_{I_Q}$. Then $m - 1 \ge |I_N|$ so we apply the inductive hypothesis to conclude that $|I_N| = m-1$ and find $s_1, \dots, s_{m-1}$ with $s_1s_2\cdots s_{m-1} = 1$ such that after reordering and unitary conjugation (in $N^t$) we have $M_i = P_i^{ts_i}$ for $1 \le i \le m-1$. 
But 
\begin{align}\label{E: gecase}
m \ge k = |I_N| + |I_Q| = m-1 + |I_Q| \implies |I_Q| = 1 \text{ and } m = k, 
\end{align}
so setting $t_m = 1/t$ and $t_i = ts_i$ for $1 \le i \le m-1$ finishes the proof of (2). 

For (3), we proceed just as for (2), except that we replace \eqref{E: gecase} by the observation that $P_m^{1/t} = M_{I_Q}$ implies $|I_Q| = 1$ when $P_m$ is assumed to be prime. 

\end{proof}

\subsection{Unique Prime Factorization for Equivalence Relations}

In order to deduce Theorem \ref{T: factorR} from Theorem \ref{T: factor}, we prove the following proposition:

\begin{repmyprop}{P: nonGamma}
Let $\R$ be a strongly ergodic countable pmp equivalence relation which is nonamenable and admits an unbounded 1-cocycle into a mixing orthogonal representation weakly contained in the regular representation. Then $L(\R)$ is prime and does not have property Gamma. 
\end{repmyprop}
\begin{proof}
That $M = L(\R)$ is prime is simply a special case of Theorem \ref{T: main}. Again, consider the $s$-malleable deformation $M \subset \tilde M$, $\{\alpha_t\}_{t \in \RR} \subset \Aut(\tilde M)$ constructed in Section \ref{S: deform}, and suppose toward a contradiction that $M$ has property Gamma. 
Then there is a sequence $\{u_n\} \in \U(M)$ with $\tau(u_n) = 0$ for all $n$ and $\|u_n x - x u_n\|_2 \to 0$ as $n \to \infty$ for each $x \in M$. 
Then for any $u \in \N_M(A)$ we have 
\begin{align*}
\|u\E_A(u_n)u^* - \E_A(u_n)\|_2 = \|\E_A(uu_nu^*) - \E_A(u_n)\|_2 \le \|uu_nu^* - u_n\|_2 \to 0 \text{ as } n \to \infty
\end{align*}
Since the sequence $\E_A(u_n)$ is bounded in norm and $M = \N_M(A)''$, it follows that $\|x \E_A(u_n) - \E_A(u_n)x\|_2 \to 0$ for each $x \in M$. Since $\R$ is strongly ergodic, it follows that $\|\E_A(u_n)\|_2 = \|\E_A(u_n) - \tau(\E_A(u_n))\|_2 \to 0$ as $n \to \infty$. 

Fix any $g \in [\R]$ with $g^2 = e$. Note $z_g = \E_A(u_g)$ is a projection given by $z_g = 1_{\{s \in X: gs = s\}} = z_{g^{-1}}$. Moreover, $u_g^*\E_A(u_g) \in A' \cap M = A \implies u_g^*z_g = \E_A(u_g^*z_g) = \E_A(u_g^*)z_g = z_g$. Hence 
\begin{align} \label{E: z_g}
\|\E_A(u_nu_g^*)z_g\|_2 = \|\E_A(u_nu_g^*z_g)\|_2 = \|\E_A(u_n)z_g\|_2 \le \|\E_A(u_n)\|_2 \to 0 \text{ as } n \to \infty
\end{align}
Moreover, since $1 - z_g = 1_{\{s \in X: gs \ne s\}}$, for any nonzero $z \le 1 - z_g$ with $u_gzu_g^* = z$, we can find nonzero $z' \le z$ such that $u_gz'u_g^* \le z - z'$ (if not we would have $u_gz'u_g^* = z'$ for all $z' \le z$ and then $z \le z_g$). 
Then because $g^2 = e$, it follows that we can find a projection $z \in A$ such that $1 - z_g = z + u_gzu_g^*$, so that 
\begin{align} \label{E: 1-z_g}
\|\E_A(u_nu_g^*)(1-z_g)\|^2_2 
&= \|\E_A(u_nu_g^*)(z + u_gzu_g^*)\|_2^2 \notag\\
&= \|\E_A(u_nu_g^*)z\|_2^2 + \|\E_A(u_nu_g^*)u_gzu_g^*\|_2^2  \notag\\
&= \|\E_A(u_nu_g^*)(z - u_gzu_g^*)\|_2^2
= \|z\E_A(u_nu_g^*) - \E_A(u_nu_g^*u_gzu_g^*)\|_2^2 \notag\\
&= \|\E_A(zu_nu_g^* - u_nzu_g^*)\|_2^2
\le \|zu_n - u_nz\|_2^2 \to 0 \text{ as } n \to \infty. 
\end{align} 

Combining \eqref{E: z_g} and \eqref{E: 1-z_g} we see that $\|\E_A(u_nu_g^*)\|_2 \to 0$ as $n \to \infty$ for each $g \in [\R]$ with $g^2 = e$. By Feldman and Moore \cite{FM75a}, we know that $(x, y) \in \R$ if and only if $y = gx$ for some $g \in [\R]$ with $g^2 = e$, so that $L(\R) = \vN{au_g : a \in A, g \in [\R], g^2 = e}$. 
It therefore follows that $\|\E_A(xu_ny)\|_2 \to 0$ as $n \to \infty$ for any $x, y \in M$. 

From the proof of Theorem \ref{T: main}, we know that ${}_ML^2(\tilde M) \ominus L^2(M)_M$ is mixing relative to $A$, so this implies that $\<u_n\delta_t(x), \delta_t(x)u_n\> \to 0$ as $n \to \infty$ for each $x \in M$. We also know that $\alpha_t \to \id$ is not uniform on $(M)_1$, and hence by part (1) of Lemma \ref{L: trans} there is $\epsilon > 0$ and sequences $\{x_k\} \subset (M)_1$, $\{t_k\} \subset \RR$, $t_k \to 0$, such that $\|\delta_{t_k}(x_k)\|_2 \ge \epsilon$ for all $k$. Then using Lemma \ref{L: trans}, for any $k$ we get
\begin{align*}
\epsilon^2 &\le \|\delta_{t_k}(x_k)\|_2^2 
= \liminf_{n \to \infty} \left[\frac{1}{2}\|[u_n, \delta_{t_k}(x_k)]\|_2^2 + \text{Re}\<u_n\delta_{t_k}(x_k), \delta_{t_k}(x_k)u_n\> \right] \\
& \le \frac{1}{2}\liminf_{n \to \infty} \left[2\|\alpha_{t_k}(u_n) - u_n\|_2 + \|[u_n, x_k]\|_2\right]^2
\le 8\liminf_{n \to \infty} \|\delta_{t_k/2}(u_n)\|_2^2
\end{align*}
Thus setting $s_k = \frac{t_k}{2}$ for each $k$ we can find $n_k \ge k$ such that $\|\delta_{s_k}(u_{n_k})\|_2 \ge \frac{\epsilon}{4}$. Then using Lemma \ref{L: trans} again, for any $x \in M$,
\begin{align*}
\|[\delta_{s_k}(u_{n_k}), x]\|_2 
&\le 2\|\alpha_{s_k}(x) - x\|_2 + \|[u_{n_k}, x]\|_2 \to 0 \text{ as } k \to \infty.
\end{align*}
Since we also have $\|x\delta_{t_k/2}(u_{n_k})\|_2 \le \|x\|_2$ for all $k$, we apply Lemma \ref{L: common} to find that ${}_ML^2(M)_M \prec {}_ML^2(\tilde M) \ominus L^2(M)_M$. But since we know from the proof of Theorem \ref{T: main} that ${}_ML^2(\tilde M) \ominus L^2(M)_M \prec {}_ML^2(M) \ten L^2(M)_M$, this implies that $M$ is amenable, a contradiction. 
\end{proof}

Combining Theorem \ref{T: factor} with Proposition \ref{P: nonGamma} and the proof of Theorem \ref{T: main}, we get Theorem \ref{T: factorR} immediately. We prove Corollary \ref{C: factorRR} below:
%
%
%

\begin{proof}[Proof of Corollary \ref{C: factorRR}] For (1), let $A_i \subset L(\R_i)$ and $B_i \subset L(S_i)$ denote the canonical Cartan algebras of the factors. By \cite{FM75b}, the hypothesis leads to a normal $*$-isomorphism $M = L(\R_1) \ten L(\R_2) \ten \dots \ten L(\R_k) \cong L(S_1) \ten L(S_2)$ which identifies $A_1 \ten A_2 \ten \cdots \ten A_k = B_1 \ten B_2$. 
Applying Theorem \ref{T: factorR}, we find $t > 0$, $u \in \U(M)$, and an integer $1 \le m < k$ such that after reordering the indices we have $uL(S_1)^tu^* = L(\R_1) \ten L(\R_2) \ten \dots \ten L(\R_m)$ and $uL(S_2)^{1/t}u^* = L(\R_{m+1}) \ten L(\R_{m+2}) \ten \dots \ten L(\R_k)$.
Setting $A = A_1 \ten \cdots \ten A_k$, we have $uB_1^tu^* \prec_M A$ (as $u^*(uB_1^tu^*)u \subset A$) which implies that $uB_1^tu^* \prec_{uL(S_1)^tu^*} A_1 \ten \cdots \ten A_m$. Indeed, if there were $\{u_n\} \subset \U(uB_1^tu^*)$ with $\|\E_{A_1 \ten \cdots \ten A_m}(xu_ny)\|_2 \to 0$ for all $x, y \in uL(S_1)^tu^*$, one can check that it would give $\|\E_{A}(xu_ny)\|_2 \to 0$ for all $x, y \in M$ as well. Then since $uB_1^tu^*$ and $A_1 \ten \cdots \ten A_m$ are both Cartan subalgebras of $uL(S_1)^tu^*$, we know as in \cite{Po01b} that there is $v_1 \in \U(uL(S_1)^tu^*)$ such that $v_1uB_1^tu^*v_1^* = A_1 \ten \cdots \ten A_m$. Thus $\text{ad } v_1u$ is an isomorphism of $L(S_1^t) \cong L(S_1)^t$ onto $L(\R_1 \times \R_2 \times \dots \times \R_m)$ which identifies $B_1^t$ and $A_1 \ten \cdots \ten A_m$. A second application of \cite{FM75b} then gives $S_1^t \cong \R_1 \times \R_2 \times \dots \times \R_m$. Similarly, one identifies $B_2^{1/t}$ and $A_{m+1} \ten \cdots \ten A_k$ to conclude that $S_2^{1/t} \cong \R_{m+1} \times \dots \times \R_k$. 

We prove (2) by induction on $k$. The case $k = 1$ follows immediately from Theorem \ref{T: main}. For $k \ge 2$, we apply (1) to find $t > 0$ and an integer $1 \le j < k$ such that after reordering indices, $S_1^t \times \cdots \times S_{m-1}^t \cong \R_1 \times \cdots \times \R_j$ and $S_m^{1/t} \cong \R_{j+1} \times \cdots \times \R_k$. Then $m - 1 \ge j$ so we apply the inductive hypothesis to conclude that $j = m-1$ and find $s_1, \dots, s_{m-1}$ with $s_1s_2\cdots s_{m-1} = 1$ such that after reordering we have $\R_i = S_i^{ts_i}$ for $1 \le i \le m-1$. 
Finally, we have $0 < k - j \le m - (m - 1) = 1$, and so $k = j + 1$, $m = k$, and $S_m^{1/t} \cong \R_m$. 
\end{proof}

\section{Application to Measure Equivalent Groups} \label{S: ME}

The tools developed in the previous sections lend themselves easily to the measure equivalence of groups, a notion first introduced by Gromov \cite{Gr91}. 
Countable groups $\Gamma_1$ and $\Gamma_2$ are called \emph{measure equivalent (ME)}, written $\Gamma_1 \stackrel{\text{ME}}{\sim} \Gamma_2$, if there is a Lebesgue measure space $(Y, \nu)$ and commuting free measure preserving actions $\Gamma_i \curvearrowright (Y, \nu)$, $i \in \{1, 2\}$, which each admit a finite measure fundamental domain. 

Measure equivalence is closely related to stable orbit equivalence. Recall that two probability measure preserving actions $\Gamma_i \curvearrowright (X_i, \mu_i)$, $i \in \{1, 2\}$, on standard probability spaces $(X_i, \mu_i)$ are \emph{stably orbit equivalent (SOE)} if for each $i \in \{1, 2\}$ we can choose a measurable subset $E_i \subset X_i$ meeting the orbit of a.e. $x \in X_i$, such that the restricted equivalence relations are isomorphic, i.e. $\R_1|_{E_1} \cong \R_2|_{E_2}$ where $\R_i = \R(\Gamma_i \on{} X_i)$ for $i \in \{1, 2\}$. 
Then $\Gamma \stackrel{\text{ME}}{\sim} \Lambda$ if and only if $\Gamma$ and $\Lambda$ admit SOE free actions. This equivalence was proved by Furman in \cite{Fu99b} where it is attributed to Zimmer and Gromov, and the form stated here (that the actions can be taken to be free) was proved in \cite{Ga00b}. 

Gaboriau showed in \cite{Ga02} that measure equivalent groups have proportional $\ell^2$ Betti numbers, i.e.,
 if $\Gamma \stackrel{\text{ME}}{\sim} \Lambda$ there is $\lambda > 0$ such that $\beta_n(\Gamma) = \lambda\beta_n(\Lambda)$ for all $n$. 
In particular, if $\beta_1(\Gamma) > 0$ then $\Gamma$ cannot be measure equivalent to a product of infinite groups (as $\beta_1 = 0$ for a product of infinite groups).
The following theorem strengthens this conclusion since we know from \cite{PT07} that if $\beta_1(\Gamma) > 0$ then $\Gamma$ is nonamenable and admits an unbounded 1-cocycle for the left regular representation (which is mixing). 

\begin{repmytheorem}{T: ME}
Let $\Gamma$ be a countable nonamenable group which admits an unbounded 1-cocycle into a mixing orthogonal representation weakly contained in the left regular representation. Then $\Gamma \stackrel{\text{ME}}{\nsim} \Gamma_1 \times \Gamma_2$ for any infinite groups $\Gamma_1, \Gamma_2$. 
\end{repmytheorem}

\begin{proof}
Suppose that $\Gamma \stackrel{\text{ME}}{\sim} \Gamma_1 \times \Gamma_2$ for groups $\Gamma_1, \Gamma_2$. Then there are SOE free actions $\Gamma \on{} (X, \mu)$ and $\Gamma_1 \times \Gamma_2 \on{} (Y, \nu)$. 
Letting
$\R = \R(\Gamma \on{} X)$ and $\R' = \R(\Gamma_1 \times \Gamma_2 \on{} Y)$, this means there are
measurable $E \subset X$, $F \subset Y$ meeting a.e. orbit and such that $\R|_E \cong \R'|_F$. 
We may assume that $\R$ and $\R'$ are ergodic, since if not, we replace $\mu|_E \cong \nu|_F$ by a measure in the ergodic decomposition of $\R|_E \cong \R'|_F$ and then extend this measure to $\tilde \mu$ on $X$ and $\tilde \nu$ on $Y$ using the fact that $E \subset X$, $F \subset Y$ meet a.e. orbit.
Then for $t_1 = \mu(E)$, $t_2 = \nu(F)$, we have $\R^{t_1} \cong (\R')^{t_2}$, and hence $\R^{t_1/t_2} \cong \R'$. 

Set $t = t_1/t_2$, $M = L(\R^t)$, and let $(X_t, \mu_t)$ denote the underlying probability space of $\R^t$. 
Since $\Gamma \on X$ is free, we see from \eqref{E: group2R} that $\R$ admits an unbounded 1-cocycle $b$ into a mixing orthogonal representation $\pi$ weakly contained in the regular representation. Let $\pi^t$ and $b^t$ be the amplifications as in \eqref{E: pib^t}.
 Then, as in Section \ref{S: deform}, we construct from $\pi^t$ and $b^t$ an imbedding $M \subset \tilde{M}$ and
 $s$-malleable deformation $\{\alpha_s\}_{s \in \RR} \subset \Aut(\tilde{M})$, $\beta \in \Aut(\tilde{M})$.  
As in the proof of Theorem \ref{T: main}, we know that $L^2(\tilde{M}) \ominus L^2({M})$ is weakly contained in the coarse $M$-$M$ bimodule and mixing relative to the abelian subalgebra $A = L^\infty(X_t)$. Thus $M$ satisfies the assumptions of Theorem \ref{T: deduce}, and we need only modify its proof slightly. 

We know that $\R$ (and hence $\R^t$) is nonamenable since $\Gamma$ is nonamenable and $\Gamma \on X$ is free. It follows that either $\Gamma_1$ or $\Gamma_2$ must be nonamenable, so assume without loss of generality that $\Gamma_2$ is nonamenable.
Since $\R^t \cong \R'$, we have an isomorphism $M \cong L^\infty(Y) \rtimes (\Gamma_1 \times \Gamma_2)$ which identifies $A = L^\infty(X_t)$ and $L^\infty(Y)$. 
We therefore consider the commuting subalgebras $L(\Gamma_1), L(\Gamma_2) \subset M$. Then just as in the proof of Theorem \ref{T: deduce}, we must have $\alpha_s \to \id$ uniformly in $\|\cdot \|_2$ on the unit ball of $L(\Gamma_1)$, since otherwise we would have
\[
_{L(\Gamma_2)}L^2(L(\Gamma_2))_{L(\Gamma_2)} 
\prec {_{L(\Gamma_2)}L^2(\widetilde{M}) \ominus L^2({M})_{L(\Gamma_2)}}
\prec {_{L(\Gamma_2)}L^2(L(\Gamma_2))\ten L^2(L(\Gamma_2))_{L(\Gamma_2)}} 
\]
contradicting the nonamenability of $\Gamma_2$.

Assuming toward a contradiction that $\Gamma_1$ is also infinite, take a sequence $\{u_{g_n}\}_{n = 1}^\infty \subset \Gamma_1$. From the freeness of the action it follows that $\lim_{n \to \infty} \|\E_A(xu_{g_n}y)\|_2 = 0$ for each $x, y \in M$. Then just as in \eqref{E: usemix}, combining the sequence $\{u_{g_n}\}$ with the mixingness of $L^2(\widetilde{M^t}) \ominus L^2({M^t})$ relative to $A$ gives $\alpha_s \to \id$ uniformly in $\|\cdot \|_2$ on the unit ball of $L(\Gamma_2)$. 

Therefore for any $\epsilon > 0$, we can find $s_0 > 0$ such that for $|s| < s_0$  we have $\|\alpha_s(x) - x\|_2 < \frac{\epsilon}{4}$ for all $x \in L(\Gamma_1) \cup L(\Gamma_2)$ with $\|x\| \le 1$. 
Then for $|s| < s_0$, the $\|\cdot\|_2$-closed convex hull $K_s$ of the set $\{\alpha_s(u_g)\alpha_s(u_h)u_g^*u_h^* : g \in \Gamma_1, h \in \Gamma_2\}$ has a unique element $k_s \in K_s$ of minimal $\|\cdot\|_2$ satisfying $\|k_s - 1\| \le \frac{\epsilon}{2}$. 


For $a \in \U(A)$ and $(g, h) \in \Gamma_1 \times \Gamma_2$ using the facts that $\alpha_s(a) = a$, $[u_g, u_h] = 0$ and $u_g, u_h \in \N_M(A)$, one can check that $\alpha_s(au_gu_h)K_s(au_gu_h)^* = K_s$. 
From the uniqueness of $k_s$ it then follows that $\alpha_s(au_gu_h)k_s(au_gu_h)^* = k_s$ and hence $\alpha_s(x)k_s = k_sx$ for all $x \in M$. Then 
\[
\|\alpha_s(x) - x\|_2 \le \|\alpha_s(x) - \alpha_s(x)k_s\|_2 + \|k_sx - x\|_2 \le 2\|k_s - 1\|_2 \le \epsilon
\]
for all $x \in (M)_1$, $|s| < s_0$. 
Thus $\alpha_s \to \id$ uniformly on $(M)_1$, which contradicts the unboundedness of $b^t$ just as in \eqref{E: contra}. 
\end{proof}

\bibliographystyle{myamsalpha.bst}
\bibliography{References}

\end{document}